\documentclass[12pt]{amsart}

\usepackage{amsmath, amssymb}
\usepackage{mathtools}
\usepackage{bbm}
\usepackage[utf8]{inputenc}
\usepackage{calrsfs}
\DeclareMathAlphabet{\pazocal}{OMS}{zplm}{m}{n}
\usepackage{overpic}
\usepackage{hyperref}
\usepackage{soul}

\makeatletter
\@namedef{subjclassname@2020}{\textup{2020} Mathematics Subject Classification}
\makeatother

\usepackage{float}

\newtheorem{theorem}{Theorem}[section]
\newtheorem{lemma}[theorem]{Lemma}
\newtheorem{prop}[theorem]{Proposition}
\newtheorem{cor}[theorem]{Corollary}

\theoremstyle{definition}
\newtheorem{defi}[theorem]{Definition}

\theoremstyle{remark}
\newtheorem{remark}[theorem]{Remark}

\numberwithin{equation}{section}

\newcommand{\R}{\mathbb R}

\newcommand{\Rd}{{\mathbb R^d}}

\newcommand{\E}{{\mathbb E}}
\newcommand{\Pb}{\mathbb P}
\newcommand{\N}{{\mathbb N}}

\newcommand{\Z}{{\mathbb Z}}
\newcommand{\m}{\text{-}}
\newcommand{\eps}{\varepsilon}
\newcommand{\ph}{\varphi}

\newcommand{\eqfd}{\overset{\mathrm{fd}}{=}}
\newcommand{\eqd}{\overset{\mathrm{d}}{=}}

\DeclareMathOperator{\Ca}{Cap}

\allowdisplaybreaks

\begin{document}

\title[Parabolic Fractal Geometry of Stable L\'{e}vy Processes with Drift]{Parabolic Fractal Geometry of Stable L\'{e}vy Processes with Drift}

\author{Peter Kern}
\address{Peter Kern, Mathematical Institute, Faculty of Mathematics and Natural Sciences, Heinrich Heine University D\"usseldorf, Universit\"atsstr. 1, D-40225 D\"usseldorf, Germany}
\email{kern@hhu.de}

\author{Leonard Pleschberger}
\address{Leonard Pleschberger, Mathematical Institute, Faculty of Mathematics and Natural Sciences, Heinrich Heine University D\"ussel\-dorf, Universit\"atsstr. 1, D-40225 D\"usseldorf, Germany}
\email{leonard.pleschberger@hhu.de}

\date{\today}

\begin{abstract}
We explicitly calculate the Hausdorff dimension of the graph and range of an isotropic stable L\'{e}vy process $X$ plus deterministic drift function $f$. For that purpose we use a restricted version of the genuine Hausdorff dimension which is called the parabolic Hausdorff dimension. It turns out that covers by parabolic cylinders are optimal for treating self-similar processes, since their distinct non-linear scaling between time and space geometrically matches the self-similarity of the processes. We provide explicit formulas for the Hausdorff dimension of the graph and the range of $X+f$. In sum the parabolic Hausdorff dimension of the drift term $f$ alone contributes to the Hausdorff dimension of $X+f$.  Furthermore, we derive some formulas and bounds for the parabolic Hausdorff dimension. 

\end{abstract}


\maketitle

\baselineskip=18pt

\section{Introduction}

Let $X = (X_t)_{t \geq 0}$ be a L\'{e}vy process in $\Rd$ which is a stochastic process on some probability space $(\Omega,\pazocal{A},\Pb)$ with the following properties:
\begin{enumerate}
\item[(i)] The process $\Pb$-almost surely starts in $0 \in \Rd$.
\item[(ii)] $X$ possesses independent increments, i.e.\ for any $0 \leq t_0 < \dots < t_n$, the random variables $X_{t_0},X_{t_1} - X_{t_0},\dots,X_{t_n}-X_{t_{n-1}}$ are independent.
\item[(iii)] $X$ has stationary increments, i.e.\ for all $t,h\geq0$ the distribution of the increment $X_{t+h}-X_t\eqd X_{h}$ does not depend on $t$, where the symbol $\eqd$ denotes equality in distribution. 
\item[(iv)] $X$ is stochastically continuous, i.e.\ $\Pb(\|X_{t+h} - X_t\| > \eps) \rightarrow 0$ as $h \rightarrow 0$ for any $t\geq0$ and $\eps > 0$.
\end{enumerate}

Additionally assuming self-similarity, the L\'evy process is called stable. In this paper we only deal with the special case of an isotropic $\alpha$-stable L\'evy process in which case the self-similarity is given by 
\begin{equation}\label{eq:X_ct}
(X_{\, c \cdot t})_{t \geq 0} \ \eqfd \ (c^{1/\alpha} \cdot X_t)_{t \geq 0}\quad \text{ for all }c>0,
\end{equation}
where $\eqfd$ denotes equality of all finite-dimensional distributions which characterise the stochastic processes in law. In this case the Hurst index $H = 1/ \alpha$ is restricted to $H\geq\frac12$, i.e. $\alpha\in(0,2]$ and the isotropic $\alpha$-stable L\'evy process is also uniquely determined in law by the characteristic function $\E\big[{e^{i\langle \xi, X_t \rangle}}\big]= e^{-t\, C\cdot \|\xi\|^\alpha}$
with L\'{e}vy exponent $\Psi(\xi) = C\cdot \|\xi\|^\alpha$ for some constant $C>0$. In case of $\alpha = 2$ we obtain Brownian motion. For details on stable L\'evy processes we refer to the monograph \cite{S99}.

The integrability of $\exp(-t\,C\cdot \|\xi\|^\alpha)$ ensures the applicability of the Fourier inversion formula. Therefore, for any $t>0$ the random variable $X_{t}$ possesses the continuous density function
\begin{equation*}\label{stabledens}
x\mapsto p(t,x) := (2\pi)^{-d}\int_\Rd e^{-i\langle \xi, x \rangle} e^{-t \Psi(\xi)}\, \textnormal{d}\xi = (2\pi)^{-d}\int_\Rd e^{-i\langle \xi, x \rangle} e^{-t\,C\cdot \|\xi\|^\alpha}\, \textnormal{d}\xi
\end{equation*}
which for $\alpha\in(0,2)$ cannot be expressed in simple terms but belongs to $C^\infty(\Rd)$ with all derivatives in $L^1(\Rd)\cap C_0(\Rd)$; see \cite{S99}. Furthermore, from the self-similarity property  (\ref{eq:X_ct}) it easily follows that
\begin{equation}\label{eq_p_selfsimilar}
p(t,x) = t^{-d/\alpha} \cdot p\left(1,\tfrac{x}{t^{1/\alpha}}\right)\quad\text{ for all $t>0$ and $x\in\Rd$.}
\end{equation}
Thus we define $p(x) := p(1,x)$ as the density at time $t=1$ and by Theorem 2.1 in \cite{BG60} we have the tail estimate
\begin{equation}\label{eq:p_tail_estimates}
p(x) \in \pazocal{O}\big(\|x\|^{-d-\alpha}\big) \quad  \text{ as } \|x\|\rightarrow \infty.
\end{equation}
This density is bounded and rotationally symmetric, i.e.\ writing $x=ry$ with $r = \|x\| > 0$ and $y = x/\|x\| \in S^{d-1}$ the density $p(x) = p(r y)$ does not depend on $y$ and due to unimodality, see \cite{S99}, $r \mapsto p(r y)$ is non-increasing.\\

Our aim is to analyse isotropic stable L\'{e}vy processes plus (arbitrary) Borel measurable drift functions by methods from fractal geometry. In particular, we determine formulas for the Hausdorff dimension of the graph and the range of an isotropic $\alpha$-stable L\'{e}vy process plus drift. Starting with Brownian motion, in the past decades much effort has been made to explicitly calculate the Hausdorff dimension of the range and the graph of stable L\'evy processes with an even more general self-similarity relation than \eqref{eq:X_ct}, e.g.\ see \cite{BG60b,BG62,PT69,Hen73,LX94,BMS,MX05,W17} in chronological order or the excellent review article \cite{X04}. 
Some of these results were extended to Markov processes, but require additional regularity assumptions for the transition probabilities in the non-homogeneous situation, see \cite{X04} and the references therein for details.
Only recently, Peres and Sousi started to deal with Hausdorff dimension results of self-similar processes with an additional drift function by considering Brownian motion \cite{PS12} and fractional Brownian motion \cite{PS16}. We will follow the method in \cite{PS16} to prove corresponding results for isotropic stable L\'evy processes. The restriction to isotropic stable L\'evy processes is due to rotational symmetry which is needed in the proof method. Compared to the method in \cite{PS16} we have to overcome with some additional issues:
\begin{enumerate}
\item An isotropic $\alpha$-stable L\'evy process for $\alpha \in(0,2)$ is a pure jump process. Hence we cannot use H\"older continuity of the sample paths to derive upper bounds for the Hausdorff dimension as in case of fractional Brownian motion in \cite{PS16}. Instead, we will use the covering Lemma of Pruitt and Taylor (Lemma 6.1 in \cite{PT69}) for the L\'evy process together with a deterministic cover of the drift function.
\item The Hurst index $H=1/\alpha$ of an isotropic $\alpha$-stable L\'evy process is restricted to $H\geq 1/2$, whereas $H\in(0,1)$ for fractional Brownian motion. It will turn out that the case $H\geq 1$ needs different arguments than the blueprint given for $H\in(0,1)$ in \cite{PS16}. The reason is that diameters of parabolic cylinders introduced in \eqref{eq:1/alpha} are determined by the size of the cylinders in the space domain for $H\in(0,1)$, whereas for $H>1$ the size of the cylinders in the time domain dominates.
\item By \eqref{eq:p_tail_estimates} the tail of the probability density of an isotropic $\alpha$-stable L\'evy process falls off as a power function, whereas the normal density of fractional Brownian motion decreases exponentially fast. The power law tails require more delicate estimates of energy kernels in the potential theoretic method to derive lower bounds of the Hausdorff dimension.
\end{enumerate}
We introduce a generalised version of the genuine Hausdorff dimension which is called the $\alpha$-parabolic Hausdorff dimension in Section 2. We also give a priori upper and lower bounds for the $\alpha$-parabolic Hausdorff dimension in terms of the genuine Hausdorff dimension. It turns out that covers by $\alpha$-parabolic cylinders are optimal for treating self-similar processes, since their distinct non-linear scaling between time and space geometrically matches the self-similarity of the processes. We provide explicit formulas for the Hausdorff dimension of the graph and the range of an isotropic $\alpha$-stable L\'evy process plus Borel measurable drift function in Section 3 and defer the proofs to Sections 4--6. In sum the $\alpha$-parabolic Hausdorff dimension of the drift term $f$ alone contributes to the Hausdorff dimension of $X+f$. We derive new formulas and estimates for the $\alpha$-parabolic Hausdorff dimension of constant functions and Hölder continuous functions in Section 7.

\section{On Parabolic Fractal Geometry}

We start with the definition of the $\alpha$-parabolic Hausdorff dimension which inheres a distinct non-linear scaling between time and space.
Throughout this paper the symbol $|\cdot|$ denotes the diameter of a set in real Euclidean space. We use the same symbol for the absolute value of reals, without fear to cause confusion, and denote by $\|\cdot\|$ the Euclidean norm of vectors in $\Rd$ or $\R^{1+d}$. For real functions $f,g$ the symbol $f\lesssim g$ denotes the existence of a constant $C\in(0,\infty)$ not depending on the variables such that $f\leq C\cdot g$ and $f\asymp g$ is short for $f\lesssim g$ together with $g\lesssim f$.

\begin{defi} \label{def:parabolic_measure}
Let $A\subseteq \R^{1+d}$ be an any set and $\alpha,\beta \in (0,\infty)$. The $\alpha$-\emph{parabolic $\beta$-Hausdorff (outer) measure} of $A$ is defined as
\begin{gather*}
\mathcal{P}^\alpha\m \pazocal{H}^\beta(A) := \lim_{\delta \downarrow 0}\ \inf \Bigg\{\sum_{k=1}^\infty |\mathsf{P}_k|^\beta: A\subseteq\bigcup_{n=1}^\infty \mathsf{P}_k,\ \mathsf{P}_k \in \mathcal{P}^\alpha,\ |\mathsf{P}_k|\leq \delta \Bigg\}, \label{def:Hausdorff_measure_rule}
\end{gather*}
where the \emph{$\alpha$-parabolic cylinders} $(\mathsf{P}_n)_{n\in\N}$ are contained in
\begin{equation}\label{eq:1/alpha}
\mathcal{P}^\alpha :=\Bigg \{ [t,t+c]\times \prod_{j=1}^d\ [x_{j},x_{j}+c^{1/\alpha}]: t,\,  x_j \in \R,\ c \in (0,1] \Bigg\}.
\end{equation}
We define the \emph{$\alpha$-parabolic Hausdorff dimension} of $A$ as
\begin{equation*}
\mathcal{P}^\alpha\m\dim A := \inf \big\{\beta>0 : \mathcal{P}^\alpha\m\pazocal{H}^\beta(A) =0 \big\} = \sup\, \big\{\beta>0 : \mathcal{P}^\alpha\m\pazocal{H}^\beta(A) = \infty \big\}.
\end{equation*}
The case $\alpha = 1$ equals the genuine Hausdorff dimension which is simply denoted by the symbol $\dim$.
\end{defi}

Let us compare the $\alpha$-parabolic Hausdorff dimension to other parabolic Hausdorff dimensions appearing in the literature. For the Gaussian case, Taylor and Watson \cite{TW85} introduced the \emph{parabolic Hausdorff dimension}  $\mathcal{P}\m\dim$ in the same way we did for $\alpha=2$ in order to determine polar sets for the heat equation. On the contrary, for $H\in(0,1]$ Peres and Sousi \cite{PS16} defined the $H$\emph{-parabolic $\beta$-Hausdorff content}
\begin{equation*}\label{def:Hausdorff_Content}
\Psi_H^\beta(A) := \inf \bigg\{\sum_{k=1}^\infty c_k^\beta: A\subseteq \bigcup_{k=1}^\infty \mathsf{P}_k,\ \mathsf{P}_k \in \mathcal{P}^{1/H} \bigg\},
\end{equation*}
where the covering sets
\begin{equation*}
\mathsf{P}_k=[t_{k},t_{k}+c_{k}]\times \prod_{j=1}^d\ [x_{j,k},x_{j,k}+c_{k}^{H}]
\end{equation*}
are from the class of $1/H$-parabolic cylinders $\mathcal{P}^{1/H}$ from \eqref{eq:1/alpha} with $|\mathsf{P}_k|\asymp c_{k}^{H}$. This results in what they call the $H$-\emph{parabolic Hausdorff dimension}
\begin{equation*}
\dim_{\Psi,H} A := \inf \big\{\beta>0: \Psi_H^\beta(A) = 0 \big\} = \sup \big\{\beta>0: \Psi_H^\beta(A) > 0 \big\}.
\end{equation*}

\begin{prop}\label{pro:P2=2P}
Let $A\subseteq \R^{1+d}$ be an arbitrary set, $\alpha = 1/H \in [1,\infty)$ and $\beta \in (0,\infty)$. Then one has
\begin{equation}\label{eq:PAlpha_PAlphaH}
\mathcal{P}^\alpha\m\dim A = (\dim_{\Psi,H} A) / H,
\end{equation}
i.e. Peres and Sousi's $H$-parabolic Hausdorff dimension differs from our $\alpha$-parabolic Hausdorff dimension by a constant factor. 
\end{prop}

\begin{proof}
First we introduce the auxiliary \emph{$\alpha$-parabolic $\beta$-Hausdorff content} 
\begin{equation*}
\Phi^\beta_\alpha (A) := \inf \Bigg\{\sum_{k=1}^\infty |\mathsf{P}_k|^\beta: A\subseteq \bigcup_{k=1}^\infty \mathsf{P}_k,\ \mathsf{P}_k \in \mathcal{P}^{\alpha} \Bigg\}.
\end{equation*}
Since for $\alpha \geq 1$ and $\mathsf{P}_k=[t_{k},t_{k}+c_{k}]\times \prod_{j=1}^d\ [x_{j,k},x_{j,k}+c_{k}^{1/\alpha}] \in \mathcal{P}^{\alpha}$ we have $|\mathsf{P}_k|^\beta \asymp \ c_k^{\beta / \alpha} =\ c_k^{H\cdot \beta}$ one has
\begin{equation}\label{eq:PhiBeta=PsiBeta/H}
\Phi^\beta_\alpha (A) = \Psi_H^{\beta\cdot H}(A).
\end{equation}
Following the arguments in Proposition 4.9 of \cite{MP10} one gets
$\mathcal{P}^\alpha\m\pazocal{H}^\beta (A)=0$ if and only if $\Phi_\alpha^\beta(A) = 0$,
which together with \eqref{eq:PhiBeta=PsiBeta/H} shows (\ref{eq:PAlpha_PAlphaH}).
\end{proof}

The $\alpha$-parabolic Hausdorff dimension fulfils the following countable stability property which easily follows from monotonicity and $\sigma$-subadditivity of the $\alpha$-parabolic Hausdorff measure as argued for the genuine Hausdorff dimension on page 29 in \cite{Fal90}.

\begin{prop}\label{prop:Countable_Stability}
For every $\alpha \in (0,\infty)$ and $(A_{k})_{k\in\N}\subseteq\R^{1+d}$ the $\alpha$-parabolic Hausdorff dimension fulfils the \emph{countable stability property}
\begin{equation*}
\mathcal{P}^\alpha\m\dim \bigcup_{k=1}^\infty A_k = \sup_{k\in\N}\ \mathcal{P}^\alpha\m\dim A_k.
\end{equation*}
\end{prop}

Moreover, we derive the following a priori estimates for the Hausdorff dimension in terms of the parabolic Hausdorff dimension. 

\begin{theorem}\label{le:General_Upper_Bound_Hausdorff}
Let $A\subseteq \R^{1+d}$ be any set.  Let $\phi_\alpha = \mathcal{P}^\alpha\m\dim A$. Then one has
\begin{equation*}
\dim A \leq
\begin{cases}
\phi_\alpha \ \wedge \ \left(\alpha \cdot \phi_\alpha + 1 - \alpha\right), & \alpha \in (0,1], \\
\phi_\alpha \ \wedge \ \left(\frac{1}{\alpha } \cdot \phi_\alpha + \big(1-\frac{1}{\alpha}\big)\cdot d\right) , & \alpha \in [1,\infty)
\end{cases}
\end{equation*}
and
\begin{equation*}
\dim A \geq
\begin{cases}
\phi_\alpha + \big(1-\frac{1}{\alpha}\big)\cdot d, & \alpha \in (0,1],\\
\phi_\alpha + 1 - \alpha, & \alpha \in [1,\infty).
\end{cases}
\end{equation*}
\end{theorem}

\begin{proof}
\emph{(i)} Let $\alpha \in (0,\infty)$. By the definition of the $\alpha$-parabolic $\beta$-Hausdorff measure there are only coverings by $\mathcal{P}^\alpha$-sets permitted. So besides $\mathcal{P}^\alpha$ there could exist more efficient covers of $A$ with respect to their shape. Therefore $\pazocal{H}^\beta(A)\leq \mathcal{P}^\alpha\m \pazocal{H}^\beta(A)$ which implies $\dim A \leq \phi_\alpha$.

\emph{(ii)} Let $\alpha \in (0,1]$ and $\eps > 0$ be arbitrary. If $\beta > \alpha \cdot \phi_\alpha + 1 - \alpha$, then $\frac{\beta}{\alpha} + ( 1 - \frac{1}{\alpha})> \phi_\alpha$. 
Hence we can cover $A$ by the $\alpha$-parabolic cylinders
\begin{equation*}
(\mathsf{P}_{k})_{k\in\N} = \Bigg(\big[t_{k}, \, t_{k} + c_k\big] \times \prod_{j=1}^d \Big[x_{j,k}, \, x_{j,k}+c_k^{1/\alpha}\Big]\Bigg)_{k\in\N} \subseteq \mathcal{P}^\alpha
\end{equation*}
with $\big|\mathsf{P}_{k}\big| \asymp c_k$ for every $k \in \N$ such that 
$\sum_{k=1}^\infty \big|\mathsf{P}_{k}\big|^{\beta/\alpha + 1 - 1/\alpha} \leq \eps$.
Each $\mathsf{P}_{k}$ can be covered by $\Big \lceil c_k^{1-1/\alpha} \Big \rceil$ hypercubes $\square_{k}$ with sidelength $c_k^{1/\alpha}$. Hence
\begin{align*}
\pazocal{H}^\beta(A) &\leq\ \sum_{k=1}^\infty \Big \lceil c_k^{1-1/\alpha} \Big \rceil \cdot \big|\square_{k}\big|^\beta \lesssim\ \sum_{k=1}^\infty c_k^{\beta/\alpha + 1 - 1/\alpha} \lesssim \sum_{k=1}^\infty |\mathsf{P}_{k}|^{\beta/ \alpha + 1 - 1/\alpha} \leq\ \eps.
\end{align*}
Since $\beta > \alpha \cdot \phi_\alpha + 1 - \alpha$ is arbitrary we have
$ \dim A \leq \alpha \cdot \phi_\alpha + 1 - \alpha$.

\emph{(iii)} Let $\alpha \in [1,\infty)$ and $\eps > 0$ be arbitrary. If
$\beta > 1/\alpha \cdot \phi_\alpha + (1-1/\alpha)\cdot d $, then we have 
$\alpha\beta + (1 - \alpha) \cdot d > \phi_\alpha$.
Hence a cover of $A$ by $\alpha$-parabolic cylinders $(\mathsf{P}_{k})_{k\in\N}$ as in part (i) now fulfills
$|\mathsf{P}_{k}| \asymp c_k^{1/\alpha}$ such that $\sum_{k=1}^\infty \big|\mathsf{P}_{k}\big|^{\alpha \beta + (1 - \alpha) \cdot d}  \leq \eps $. Each $\mathsf{P}_{k}$ can be covered by $\Big\lceil c_k^{1/\alpha-1}\Big\rceil^d$ hypercubes $\square_{k}$ with sidelength $c_{k}$. Then
\begin{align*}
\pazocal{H}^\beta(A) &\leq\ \sum_{k=1}^\infty \Big\lceil c_k^{1/\alpha - 1} \Big\rceil^d \cdot \big|\square_{c_k}\big|^\beta \lesssim\ \sum_{k=1}^\infty \big(c_k^{1/\alpha}\big)^{\alpha \beta + (1 - \alpha) \cdot d} \lesssim \sum_{k=1}^\infty |\mathsf{P}_{c_k^{1/ \alpha}}|^{\alpha\beta + (1-\alpha)\cdot d} \leq\ \eps.
\end{align*}
Since $\beta > 1/\alpha \cdot \phi_\alpha + (1 - 1/\alpha) \cdot d$ is arbitrary we have $
\dim A \leq \frac{1}{\alpha} \cdot \phi_\alpha + \big(1-\frac{1}{\alpha}\big)\cdot d $.

\emph{(iv)} Let $\alpha \in (0,1]$. Further, let $\beta > \dim A$ and $\eps > 0 $ be arbitrary. Then we can cover $A$ with hypercubes
\begin{equation*}
\big(\square_{k}\big)_{k\in\N} = \Bigg(\big[t_{k},\, t_{k} + c_k\big] \times \prod_{j=1}^d \big[x_{j,k},\, x_{j,k}+c_k \big]\Bigg)_{k\in\N} \subseteq \mathcal{P}^1
\end{equation*}
of sidelength $c_k \leq 1$ such that  $\sum_{k=1}^\infty \big|\square_{k}\big|^{\beta} \leq \eps$. Each $\square_{k}$ can be covered by $\Big \lceil c_k^{1-1/\alpha} \Big \rceil^d$ $\alpha$-parabolic cylinders
\begin{equation*}
(\mathsf{P}_{k})_{k \in \N} = \Bigg(\big[t_{k},\, t_{k} + c_k\big] \times \prod_{j=1}^d \Big[x_{j,k},\, x_{j,k}+c_k^{1/\alpha}\Big]\Bigg)_{k \in \N} \subseteq \mathcal{P}^\alpha
\end{equation*}
with $|P_{k}|\asymp c_k$. By choosing $\gamma = \beta + (1/\alpha - 1)\cdot d$ one has
\begin{align*}
\mathcal{P}^\alpha\m\pazocal{H}^\gamma(A) \leq \sum_{k=1}^\infty \Big\lceil c_k^{1 - 1/\alpha} \Big\rceil^d \cdot \big|\mathsf{P}_{k}\big|^\gamma \lesssim  \sum_{k=1}^\infty c_k^{(1-1/\alpha)d + \gamma} = \sum_{k=1}^\infty c_k^{\beta} \lesssim \sum_{k=1}^\infty |\square_{k}|^{\beta} \leq \eps.
\end{align*}
Since $\beta > \dim A$ is arbitrary, one has $\mathcal{P}^\alpha\m\dim A \leq \dim A + \big(\frac{1}{\alpha} - 1\big)\cdot d$.

\emph{(v)} Let $\alpha \in [1,\infty)$. 
Each $\square_{k}$ from part \emph{(iv)} can be covered by $\big\lceil c_k^{1-\alpha} \big \rceil$ $\alpha$-parabolic cylinders
\begin{equation*}
(\mathsf{P}_{k})_{k \in \N} = \Bigg(\Big[t_{k},\, t_{k} + c_k^\alpha\Big] \times \prod_{j=1}^d \big[x_{j,k},\, x_{j,k} + c_k\big]\Bigg) \subseteq \mathcal{P}^\alpha
\end{equation*}
with $|P_{k}| \asymp c_k$. By choosing $\gamma = \beta + \alpha - 1$, with similar calculations as above we get $\mathcal{P}^\alpha\m\pazocal{H}^\gamma(A)\lesssim\eps$.
Since $\beta > \dim A$ is arbitrary, one has $\mathcal{P}^\alpha\m\dim A \leq \dim A + \alpha - 1$ and the theorem is proven.
\end{proof}

\section{Main Results}

So far our considerations regarding the parabolic Hausdorff dimension were purely of geometric nature. Now we will apply it to stochastic processes. We unite the cogitations of the following sections and begin with the Hausdorff dimension of the graph $\pazocal{G}_T(X+f) =\{(t,X_{t}+f(t)):t\in T\}$ of an isotropic stable L\'{e}vy process $X$ plus Borel measurable drift function $f$.

\begin{theorem} \label{thm:Main_Theorem_Hausdorff}
Let $T\subseteq \R_+$ be a Borel set and $\alpha \in (0,2]$. Let $X=(X_t)_{t\geq0}$ be an isotropic $\alpha$-stable L\'{e}vy process in $\R^{d}$.  Further, let $f:T\rightarrow \R^d$ be a Borel measurable function. Let $\ph_\alpha = \mathcal{P}^\alpha\m\dim \pazocal{G}_T(f)$. Then one $\Pb$-almost surely has
\begin{equation*}
\dim \pazocal{G}_T(X+f) =
\begin{cases}
\ph_1, & \alpha \in (0,1],\\
\ph_\alpha \ \wedge \ \left(\frac{1}{\alpha} \cdot \ph_\alpha + \big(1-\frac{1}{\alpha}\big)\cdot d\right), & \alpha \in [1,2].
\end{cases}
\end{equation*}
\end{theorem}

\begin{proof}
The Gaussian case where $\alpha=2$ follows from Theorem 1.2 in \cite{PS16} together with Proposition \ref{pro:P2=2P}. The other cases will follow by the combination of Corollary \ref{cor_Upper_bound} with Theorem \ref{le:EnergyEstimates} for drift functions $f$ with $\|f(t)-f(s)\| \leq 1$ for all $s, t \in T$. For the treatment of arbitrary drift functions $f$ we write $T=\bigcup_{z\in(\Z/2)^{d}}T_{z}$ where $T_z := \{t \in T: \|f(t)-z\|\leq\tfrac12\}$.
The claim now follows easily by using the countable stability property from Proposition \ref{prop:Countable_Stability}.
\end{proof}

The formula for the Hausdorff dimension of the range $\pazocal{R}_T(X+f) =\{X_{t}+f(t):t\in T\}$ of an isotropic stable L\'{e}vy process $X$ plus Borel measurable drift function $f$ reads as follows. 

\begin{theorem}\label{thm:Main_theorem_range}
Let $T\subseteq \R_+$ be a Borel set and $\alpha \in (0,2]$. Let $X=(X_t)_{t\geq0}$ be an isotropic $\alpha$-stable L\'{e}vy process in $\R^{d}$ and $f:T\rightarrow \Rd$ be a Borel measurable function. Let $\ph_\alpha = \mathcal{P}^\alpha\m\dim \pazocal{G}_T(f)$. Then one $\Pb$-almost surely has
\begin{equation*}
\dim \pazocal{R}_T(X+f) =
\begin{cases}
\left(\alpha \cdot \ph_\alpha\right) \ \wedge \ d, & \alpha \in (0,1],\\
\ph_\alpha \ \wedge \ d, & \alpha \in [1,2].
\end{cases}
\end{equation*}
\end{theorem}

\begin{proof}
The Gaussian case where $\alpha=2$ follows from Theorem 1.2 in \cite{PS16}. The rest will follow by Theorem \ref{thm:Range_dim_X+f_leq_dim_f} and Theorem \ref{thm:Main_Theorem_Hausdorff_Range} analogously to the proof of Theorem \ref{thm:Main_Theorem_Hausdorff}.
\end{proof}

Our main theorems imply an improvement of Theorem \ref{le:General_Upper_Bound_Hausdorff} for graph sets.

\begin{cor}\label{cor:improvement_upper_bound}
Let $\ph_\alpha = \mathcal{P}^\alpha\m\dim \pazocal{G}_T(f)$. In case of $\alpha \in (0,1]$ one has
\begin{equation*}
\ph_1 \geq \left(\alpha \cdot \ph_\alpha\right) \ \vee \ \left(\ph_\alpha + \left(1-\tfrac{1}{\alpha}\right)\cdot d\right).
\end{equation*}
\end{cor}

\begin{proof}
For $\alpha \in (0,1]$ the combination of Theorem \ref{thm:Main_Theorem_Hausdorff} and Theorem \ref{thm:Main_theorem_range} directly yields
\begin{equation*}
\ph_1  = \dim \pazocal{G}_T(f) \geq \dim \pazocal{G}_T(X + f) \geq \dim \pazocal{R}_T(X+f) = \left(\alpha \cdot \ph_\alpha\right) \ \wedge \ d.
\end{equation*}
Furthermore, we have 
\begin{equation*}
\alpha \cdot \ph_\alpha \geq \ph_\alpha + \left(1-\tfrac{1}{\alpha}\right)\cdot d \quad \text{ if and only if }\quad \alpha \cdot \ph_\alpha \leq d
\end{equation*}
and
\begin{equation*}
d \geq \ph_\alpha + \left(1-\tfrac{1}{\alpha}\right)\cdot d \quad\text{ if and only if } \quad \alpha \cdot \ph_\alpha \leq d
\end{equation*}
which proves the claim.
\end{proof}

\section{Graph: Upper Bound via Geometric Measure Theory}

We calculate an upper bound for the Hausdorff dimension of the graph of an isotropic stable L\'{e}vy process $X$ plus drift function by means of an efficient covering.

\begin{theorem}\label{thm:dim_X+f_leq_dim_f}
Let $T\subseteq \R_+$ be any set and $\alpha \in (0,2]$. Let $X=(X_t)_{t\geq0}$ be an isotropic $\alpha$-stable L\'{e}vy process in $\R^{d}$ and $f:T\rightarrow \Rd$ be any function. Furthermore let $\ph_\alpha = \mathcal{P}^\alpha\m\dim \pazocal{G}_T(f)$ be the $\alpha$-parabolic Hausdorff dimension of the graph of $f$ over $T$. 
Then for $\alpha \in (0,1]$ one $\Pb$-almost surely has
\begin{align*}
\dim \pazocal{G}_T(X+f) \leq \dim \pazocal{G}_T(f) = \ph_1,
\end{align*}
and for $\alpha \in [1,2]$ one $\Pb$-almost surely has
\begin{align*}
\mathcal{P}^\alpha\m\dim \pazocal{G}_T(X+f) \leq \mathcal{P}^\alpha\m\dim \pazocal{G}_T(f) = \ph_\alpha.
\end{align*}
\end{theorem}

\begin{proof} 
\textit{(i)} 
Let $\alpha \in (0,1]$, $\beta = \ph_1$ and let $\delta,\, \eps>0$ be arbitrary. Then $\pazocal{G}_T(f)$ can be covered by hypercubes
\begin{equation*}
(\square_{k})_{k \in \N} = \left(\big[t_k,\, t_k + c_k \big] \times \prod_{i=1}^d \Big[x_{i,k},\, x_{i,k}+c_k \Big] \right)_{k\in\N} \subseteq \mathcal{P}^1
\end{equation*}
such that $\sum_{k=1}^\infty \big|\square_{k}\big|^{\beta + \delta} \lesssim \sum_{k=1}^\infty c_k^{\, \beta + \delta} \leq \eps$. Let $M_k(\omega)$ be the random number of a fixed $2^d$-nested collection of hypercubes (see Lemma 6.1 in \cite{PT69} for the definition) with sidelength $c_k^{1/\alpha}$ that the path $t\mapsto X_t(\omega)$ hits at some time $t\in[t_k,\, t_k + c_k]$. Let $\bigcup_{k\in\N}\ \mathsf{P}_{k}(\omega) \supseteq \pazocal{G}_T(X(\omega))$ with
\begin{equation*}
\big(\mathsf{P}_{k}(\omega)\big)_{k \in \N} = \Bigg( \big[t_k,\, t_k + c_k\big] \times \bigcup_{j=1}^{M_k(\omega)} \prod_{i=1}^d \Big[\xi_{i,j,k}(\omega),\,\xi_{i,j,k}(\omega)+c_k^{1/\alpha}\Big]\Bigg)_{k \in \N}
\end{equation*}
being a corresponding random parabolic cover of the graph of this path. Then for all $t\in[t_k,\, t_k+c_k]$ there exists $j\in\{1,\dots,M_k(\omega)\}$ such that for the $i$-th component of $X+f$ we have
\begin{equation*}
\xi_{i,j,k}(\omega)+x_{i,k} \leq X_t^{(i)}(\omega)+f_{i}(t) \leq \xi_{i,j,k}(\omega)+x_{i,k} + c_k^{1/\alpha} + c_k \leq \xi_{i,j,k}(\omega)+x_{i,k} + 2 c_k.
\end{equation*}
Hence we obtain a random cover $\bigcup_{k\in\N}\ \widetilde{\square}_{k}(\omega) \supseteq \pazocal{G}_T(X(\omega)+f)$ where
\begin{align*}
\widetilde{\square}_{k}(\omega) = \big[t_k,\, t_k + c_k\big] \times \bigcup_{j=1}^{M_k(\omega)} \prod_{i=1}^d & \left( \Big[\xi_{i,j,k}(\omega)+x_{i,k},\, \xi_{i,j,k}(\omega)+x_{i,k}+c_k\Big]\right.\\
\cup\ & \, \left.\Big[\xi_{i,j,k}(\omega)+x_{i,k}+c_k,\, \xi_{i,j,k}(\omega)+x_{i,k}+2c_k\Big]\right)
\end{align*}
This is a union of $M_k(\omega)\cdot 2^d$ sets with diameter $\sqrt{d+1} \cdot c_k$. An application of Pruitt and Taylor's covering Lemma 6.1 in \cite{PT69} and Lemma 3.4 in \cite{MX05} shows that for all $\delta' > 0$ one has
\begin{equation*}
\E[M_k] \lesssim \frac{c_k}{\E\Big[T\big(c_k^{1/\alpha}/3,c_k\big)\Big]} \lesssim c_k^{-\delta'/\alpha},
\end{equation*}
where $T\big(c_k^{1/\alpha}/3,c_k\big)$ is the sojourn time of the process $(X_t)_{t\in[0,c_k]}$ in a ball of radius $c_k^{1/\alpha}/3$ centred at the origin. 
Hence we get for $\eps' = \delta + \delta'/\alpha > 0$
\begin{equation*}
\E\Big[\pazocal{H}^{\beta+\eps'}(\pazocal{G}_T(X+f))\Big]
\leq \E\Bigg[\sum_{k=1}^\infty |\widetilde{\square}_{c_k}|^{\beta+\eps'} \Bigg]
\lesssim \sum_{k=1}^\infty \E[M_k]\cdot c_k^{\beta + \eps'} \lesssim  \sum_{k=1}^\infty  c_k^{\beta + \delta} \leq \eps.
\end{equation*}
Since $\eps,\eps' > 0$ are arbitrary, we get for all $\alpha \in (0,1]$ and $\beta'>\beta$
\begin{equation*}
\E\Big[\pazocal{H}^{\beta'}(\pazocal{G}_T(X+f))\Big]=0
\end{equation*}
which implies
$\pazocal{H}^{\beta'}(\pazocal{G}_T(X+f)) =0\quad  \Pb\text{-almost surely}$.
Since $\beta'>\beta$ is arbitrary we finally get $
\dim \pazocal{G}_T(X+f) \leq \beta = \ph_1\,\,\Pb\text{-almost surely}$.

\textit{(ii)} Let $\alpha \in [1,2]$, $\beta = \mathcal{P}^\alpha\m\dim \pazocal{G}_T(f)$ and let $\eps,\delta>0$ be arbitrary. Then $\pazocal{G}_T(f)$ can be covered by $\alpha$-parabolic cylinders
\begin{equation*}
(\mathsf{P}_{k})_{k \in \N} = \left(\big[t_k,\, t_k + c_k \big] \times \prod_{i=1}^d \Big[x_{i,k},\, x_{i,k}+c_k^{1/\alpha} \Big] \right)_{k\in\N} \subseteq \mathcal{P}^\alpha 
\end{equation*}
such that $\sum_{k=1}^\infty \big|\mathsf{P}_{k}\big|^{\beta + \delta} \lesssim \sum_{k=1}^\infty  c_k^{(\beta + \delta)/\alpha} \leq \eps$. Let $M_k(\omega)$ be the random number of a fixed $2^d$-nested collection of hypercubes with sidelength $c_k^{1/\alpha}$ that the path $t\mapsto X_t(\omega)$ hits at some time $t\in[t_k,\, t_k + c_k]$. As in part \emph{(i)}, we obtain a random parabolic cover $\bigcup_{k\in\N}\ \widetilde{\mathsf{P}}_{k}(\omega) \supseteq \pazocal{G}_T(X(\omega)+f)$ where
\begin{align*}
\widetilde{\mathsf{P}}_{k}(\omega) = \big[t_k,\, t_k + c_k\big] \times \bigcup_{j=1}^{M_k(\omega)} \prod_{i=1}^d & \left( \Big[\xi_{i,j,k}(\omega)+x_{i,k},\, \xi_{i,j,k}(\omega)+x_{i,k}+c_k^{1/\alpha}\Big]\right.\\
\cup\ & \, \left.\Big[\xi_{i,j,k}(\omega)+x_{i,k}+c_k^{1/\alpha},\, \xi_{i,j,k}(\omega)+x_{i,k}+2c_k^{1/\alpha}\Big]\right).
\end{align*}
This is a union of $M_k(\omega)\cdot 2^d$ sets with diameter ${|\widetilde{\mathsf{P}}_{k}(\omega)|\lesssim} c_k^{1/\alpha}$. 
As in part \emph{(i)} we get $\E[M_k]\lesssim c_k^{-\delta'/\alpha}$.
Hence we get for $\eps' = \delta + \delta' > 0$ with the similar calculations as above
\begin{equation*}
\E\Big[\mathcal{P}^\alpha\m\pazocal{H}^{\beta+\eps'}(\pazocal{G}_T(X+f))\Big]
\lesssim \sum_{k=1}^\infty \E[M_k]\cdot c_k^{(\beta+\eps')/\alpha}
 \lesssim \sum_{k=1}^\infty c_k^{(\beta + \delta) / \alpha}
\leq \eps.
\end{equation*}
Since $\eps, \eps' > 0$ are arbitrary, as in part \textit{(i)} we finally get
\begin{equation*}
\mathcal{P}^\alpha\m\dim \pazocal{G}_T(X+f) \leq \beta = \mathcal{P}^\alpha\m\dim \pazocal{G}_T(f)\quad\Pb\text{-almost surely},
\end{equation*}
as claimed.
\end{proof}

\begin{cor}\label{cor_Upper_bound}
Let $T\subseteq \R_+$ be any set and $\alpha \in (0,2]$.  Let $X=(X_t)_{t\geq0}$ be an isotropic $\alpha$-stable L\'{e}vy process in $\R^{d}$ and $f:T\rightarrow \Rd$ be any function. Furthermore let $\ph_\alpha = \mathcal{P}^\alpha\m\dim \pazocal{G}_T(f)$. Then one $\Pb$-almost surely has
\begin{equation*}
\dim \pazocal{G}_T(X+f) \leq
\begin{cases}
\ph_1, & \alpha \in (0,1],\\
\ph_\alpha \ \wedge \ \left(\frac{1}{\alpha} \cdot \ph_\alpha + \big(1-\frac{1}{\alpha}\big)\cdot d\right), & \alpha \in [1,2].
\end{cases}
\end{equation*}
\end{cor}

\begin{proof}
The Gaussian case $\alpha=2$ follows from Corollary 2.3 in \cite{PS16} and Proposition \ref{pro:P2=2P}. The rest follows directly from Theorem \ref{le:General_Upper_Bound_Hausdorff} and Theorem \ref{thm:dim_X+f_leq_dim_f}.
\end{proof}

\section{Graph: Lower Bound via Potential Theory}

Next we want to calculate a lower bound for the Hausdorff dimension of isotropic stable L\'{e}vy processes with drift. This will be accomplished by the energy method, see Section 4.3 in \cite{MP10}. This method makes use of the Lebesgue integral. Hence for the first time we have to impose restrictions on the domain $T\subseteq \R_+$ and the drift function $f:T \rightarrow \Rd$ with regard to their measurability. For a Borel-measurable function it is well known that the graph is always a Borel set, whereas the range is not necessarily a Borel set, but belongs to the Suslin sets. Suslin sets (also called analytic sets)  superceed the Borel sets and can be represented as the image of a Borel set under a continuous mapping. For details on Suslin sets we refer to section 39 in \cite{J03}. We introduce some notions from potential theory in this slighty more general setting, to be also applicable for the range in Section 6.

\begin{defi}\label{def:energy}
Let $A \subseteq \R^{1+d}$ be a Suslin set and $\mu$ be a probability measure supported on $A$, i.e. $\mu\in\pazocal{M}^1(A)$. Further, let $K:\R^{1 + d} \rightarrow [0,\infty]$ be a Lebesgue measurable function which is called the \emph{difference kernel}.
The \emph{K-energy of a probability measure $\mu$} is defined to be
\begin{equation*}
\pazocal{E}_K(\mu):=\int_A \int_A K(t-s,x-y)\, \textnormal{d}\mu(t,x)\, \textnormal{d}\mu(s,y)
\end{equation*}
and the \emph{equilibrium value of A} is defined as $\pazocal{E}^*_K := \inf_{\mu \in \pazocal{M}^1(A)} \pazocal{E}_K(\mu)$.
We define the \emph{K-capacity of $A$} as
\begin{equation*}
\Ca_K(A) := \frac{1}{\pazocal{E}^*_K}.
\end{equation*}
Whenever the kernel has the form $K(t,x)=\|(t,x)\|^{-\beta}$, we write $\pazocal{E}_\beta(\mu)$ for $\pazocal{E}_K(\mu)$ and $\Ca_\beta(A)$ for $\Ca_K(A)$ and we refer to them as the \emph{$\beta$-energy of a probability measure $\mu$} and the \emph{Riesz $\beta$-capacity of $A$}, respectively. Next we state Frostman's theorem.
\end{defi}

\begin{theorem}\label{th:Frostman's_Theorem} Let $\alpha > 0$. For any Suslin set $A\subseteq \R^{1 + d}$ one has
\begin{equation*}
\mathcal{P}^\alpha\m\dim A \geq \dim A = \sup\{\beta: \, \Ca_\beta(A)>0 \}.
\end{equation*}
\end{theorem}

\begin{proof}
This follows from the first assertion in the proof of Theorem \ref{le:General_Upper_Bound_Hausdorff} together with Appendix B in \cite{BP17}.
\end{proof}

The next lemma shows that we can work with an energy integral where the stable process $X$ is transformed into the kernel. 

\begin{lemma}\label{le:converted_kernel}
Let $T \subseteq\R_+$ be a Borel set and $\alpha\in (0,2]$. Let $X=(X_t)_{t\geq0}$ be a stochastic process in $\R^{d}$ with stationary increments and $f:\R_+\rightarrow\Rd$ be a Borel measurable function.  Define the difference kernel 
$$K^\beta(t,x) := \E \big[\|(t, \mathrm{sign }(t) \cdot X_{|t|} + x)\|^{-\beta}\big].$$ 
Then from $\Ca_{K^\beta}(\pazocal{G}_T(f)) > 0$
it follows that $\Pb$-almost surely $\Ca_\beta(\pazocal{G}_T(X+f)) > 0$ holds.
Hence $\pazocal{E}_{K^\beta}(\mu) < \infty$ for some probability measure $\mu\in\pazocal{M}^1(\pazocal{G}_T(f))$ implies 
\begin{equation*}
\dim \pazocal{G}_T(X+f) \geq \beta\quad\Pb \text{-almost surely}.
\end{equation*}
\end{lemma}

\begin{proof}
For every $\omega \in \Omega$, the pathwise bijection $(t,f(t)) \in \pazocal{G}_T(f)$ 
if and only if $(t,X_t(\omega)+f(t) ) \in \pazocal{G}_T(X_t(\omega)+f)$
yields the existence of some random probability measure $\nu_\omega \in \pazocal{M}^1(\pazocal{G}_T(X(\omega)+f))$ with $\nu_\omega (\widetilde{A}_\omega) = \mu (A)$ for all Borel sets $A \subseteq \pazocal{G}_T(f)$ where $\widetilde{A}_\omega := \{ (t,x+X_t(\omega)):\ (t,x)\in A  \}$. Therefore, Tonelli's theorem and the stationarity of the increments of $X$ yield
\begin{align*}
&\E\big[\pazocal{E}_\beta (\nu_\omega)\big]
=\ \E\bigg[\int_{\pazocal{G}_T(X(\omega)+f)}  \int_{\pazocal{G}_T(X(\omega)+f)} \|(t-s,x-y)\|^{-\beta}\, \textnormal{d}\nu_\omega(t,x)\, \textnormal{d}\nu_\omega(s,y)\bigg]\\
&=\ \E\bigg[\int_{\pazocal{G}_T(f)} \int_{\pazocal{G}_T(f)} \|(t-s,x + X_t(\omega)-(y + X_s(\omega)))\|^{-\beta}\, \textnormal{d}\mu(t,x)\, \textnormal{d}\mu(s,y)\bigg]\\
&=\ \int_{\pazocal{G}_T(f)} \int_{\pazocal{G}_T(f)} \E\big[\|(t-s,  X_t(\omega) - X_s(\omega) +x- y)\|^{-\beta}\big]\, \textnormal{d}\mu(t,x)\, \textnormal{d}\mu(s,y)\\
&=\ \int_{\pazocal{G}_T(f)} \int_{\pazocal{G}_T(f)} \E\big[\|(t-s, \mathrm{ sign }(t-s) \cdot  X_{|t-s|}(\omega) + x - y)\|^{-\beta}\big]\, \textnormal{d}\mu(t,x)\, \textnormal{d}\mu(s,y)\\[7pt]
&=\ \pazocal{E}_{K^\beta} (\mu)
\end{align*}
By assumption,  there exists $\mu \in \pazocal{M}^1(\pazocal{G}_T(f))$ such that $\pazocal{E}_{K^\beta} (\mu)< \infty$, therefore $\pazocal{E}_\beta (\nu_\omega) < \infty\,\,\Pb\text{-almost surely}.$
The rest of the claim follows by Frostman's theorem \ref{th:Frostman's_Theorem}.
\end{proof}

Frostman's lemma provides the suitable candidate for the probability measure $\mu$. We give a parabolic version of it.

\begin{theorem}\label{th:FrostmansLemma}
Let $A \subseteq \R^{1 + d}$ be a Borel set. If $\mathcal{P}^{\alpha}\m\dim A > \beta,$ then there exists $\mu \in \pazocal{M}^1(A)$ such that we have
\begin{equation*}
\mu\Bigg(\big[t,\, t+c\big]\times \prod_{i=1}^d \Big[x_i,\, x_i+c^{1/\alpha}\Big]\Bigg) \lesssim 
\begin{cases}
c^{\, \beta}, & \alpha \in (0,1],\\
c^{\, \beta/ \alpha}, & \alpha \in [1,\infty)
\end{cases}
\end{equation*} 
for every $c \in (0,1]$ and $t,  x_1,\dots,x_d \in \R$. 
\end{theorem}

\begin{proof} The parabolic case can easily be proven along the lines of the classical case, see e.g. section 4.4 in \cite{MP10}.
\end{proof}

The following lemma is a refinement of (2.7) in \cite{PS16}.

\begin{lemma}\label{le:density_monotocity}
Let $t \in \R$ be fixed and $h: \Rd \rightarrow \R,\ h(x) = \|(t,x)\|^{-\beta}=(t^2 + \|x\|^2)^{-\beta/2}$. Then $h$ is rotationally symmetric and the mapping $r \mapsto h(r\cdot y)$ is non-increasing for $r=\|x\|$ and does not depend on $y=x/\|x\| \in S^{d-1}$. Further, let $p:\Rd \rightarrow \R$ be a rotationally symmetric function such that also $r \mapsto p(r\cdot y)$ is non-increasing for $r=\|x\|$ and $y=x/\|x\| \in S^{d-1}$. Then for all $u \in \Rd$ we have
\begin{equation*}
\int_\Rd h(x+u) \cdot p(x) \, \textnormal{d}x \lesssim \int_\Rd h(x)  \cdot p(x) \, \textnormal{d}x,
\end{equation*}
provided that the integrals exist.
\end{lemma}

\begin{proof}
The first part is obvious. Further, by monotonicity we have
\begin{align*}
& \int_\Rd h(x+u) \cdot p(x) \, \textnormal{d}x\\
&=\ \int_{\{ \|x\| < \|x+u\| \}} \underbrace{h( x+u)}_{\leq h(x)} \cdot \ p(x) \, \textnormal{d}x + \int_{\{ \|x\| \geq \|x+u\|\}} h(x+u) \cdot \underbrace{p(x)}_{\leq p(x+u)}  \, \textnormal{d}x\\
&\leq\ 2 \int_\Rd h(x) \cdot p(x)  \, \textnormal{d}x,
\end{align*}
as claimed.
\end{proof}

Inspired by Lemma 2.5 in \cite{PS16}, we give a priori estimates for the difference kernel $K^\beta = \E \big[\|(t,\mathrm{ sign }(t) \cdot X_{|t|}(\omega) + x)\|^{-\beta}\big]$ from Lemma \ref{le:converted_kernel} that will later turn out to provide appropriate estimates of the energy integral.

\begin{lemma}\label{le:kernel_estimate}
Let $\alpha \in (0,2)$ and $X=(X_t)_{t \geq 0}$ be an isotropic $\alpha$-stable L\'{e}vy process in $\Rd$. Let $\beta \geq 0$ and $\tau\in\R$, $\delta\in\Rd$ be such that $|\tau|\in(0,1]$, $\|\delta\| \in [0,1]$.  Then for the difference kernel $K^\beta(\tau,\delta) := \E\big[\|(\tau, \mathrm{ sign }(\tau) \cdot X_{|\tau|} + \delta)\|^{-\beta}\big]$ from Lemma \ref{le:converted_kernel} one has
\begin{align}\label{eq:Kernel_tau}
K^\beta(\tau,\delta)
\lesssim
\begin{cases}
|\tau|^{-\beta}, \\
|\tau|^{-\beta/\alpha}, & \text{ for } \beta < d,\\
|\tau|^{(1-1/\alpha) d - \beta}, &\text{ for } \beta > d.
\end{cases}
\end{align}
and 
\begin{align}\label{eq:Kernel_delta}
K^\beta(\tau,\delta)
\lesssim
\begin{cases}
\|\delta\|^{-\beta}, & \text{ for } \alpha \in (0,1], |\tau| \leq \|\delta\|,\\
\|\delta\|^{-\beta}, &  \text{ for } \alpha \in [1,2), \beta < d, |\tau| \leq \|\delta\|^\alpha,\\
\|\delta\|^{(\alpha - 1)d - \alpha\beta}, & \text{ for } \alpha \in [1,2),\, \beta > d, |\tau| \leq \|\delta\|^\alpha.
\end{cases}
\end{align}
\end{lemma}

\begin{proof} 
Let $p(x)$ denote the density function of $X_1 \eqd |\tau|^{-1/\alpha} X_{|\tau|}$. We define rescaled increments $\widetilde{\tau} := \tau/|\tau|^{1/\alpha} $ and $\widetilde{\delta} := \delta/|\tau|^{1/\alpha}$. Trivial estimation always yields
\begin{equation*}
\E\big[\|(\tau, \mathrm{ sign }(\tau) \cdot X_{|\tau|} + \delta)\|^{-\beta}\big] \leq |\tau|^{-\beta}.
\end{equation*}

The self-similarity of the stable L\'{e}vy process and Lemma \ref{le:density_monotocity} yield
\begin{equation}\label{kernel-estimate:1}\begin{split}
& \E\big[\|(\tau, \mathrm{ sign }(\tau) \cdot X_{|\tau|} + \delta)\|^{-\beta}\big] 
= |\tau|^{-\beta/\alpha} \int_\Rd \|(\widetilde\tau,\mathrm{ sign }(\tau) \cdot x+\widetilde{\delta})\|^{-\beta} \cdot p(x)\, \textnormal{d}x\\
& \quad= |\tau|^{-\beta/\alpha} \int_\Rd \|(|\widetilde\tau|, x+\mathrm{ sign }(\tau)\cdot\widetilde{\delta})\|^{-\beta} \cdot p(x)\, \textnormal{d}x\\
& \quad \lesssim |\tau|^{-\beta/\alpha} \int_\Rd \|(|\widetilde\tau|,x)\|^{-\beta} \cdot p(x) \, \textnormal{d}x.
\end{split}\end{equation}
Let $\beta < d$. Then by \eqref{kernel-estimate:1} we get
\begin{align*}
& \E\big[\|(\tau, \mathrm{ sign }(\tau) \cdot X_{|\tau|} + \delta)\|^{-\beta}\big] \lesssim |\tau|^{-\beta/\alpha} \int_\Rd \|x\|^{-\beta} \cdot p(x) \, \textnormal{d}x\\
& \quad \lesssim |\tau|^{-\beta / \alpha} \cdot \E\big[\|X_1\|^{-\beta}\big]
 \lesssim |\tau|^{-\beta / \alpha},
\end{align*}
since negative moments of order $\beta<d$ exist; see Lemma 3.1 in \cite{BMS}. 

Let $\beta > d$.  Then by \eqref{kernel-estimate:1} one has using the volume of a ball with radius $\widetilde{\tau}$
\begin{align*}
& \E\big[\|(\tau, \mathrm{ sign }(\tau) \cdot X_{|\tau|}+\delta)\|^{-\beta}\big] \lesssim |\tau|^{-\beta/\alpha} \int_\Rd \|(|\widetilde{\tau}|,x)\|^{-\beta} \cdot p(x)\, \textnormal{d}x\\
& \quad \leq |\tau|^{-\beta/\alpha} \bigg( \int_{\{\|x\| < |\widetilde{\tau}|  \}}  |\widetilde{\tau}|^{-\beta} \cdot p(x)\, \textnormal{d}x + \int_{\{\|x\| \geq |\widetilde{\tau}|  \}} \|x\|^{-\beta} \cdot p(x)\, \textnormal{d}x \bigg)\\
& \quad \leq |\tau|^{-\beta/\alpha} \bigg( |\widetilde{\tau}|^{\, d - \beta} + \int_{\{ |\widetilde{\tau}| \leq \|x\| \leq 1  \}} \|x\|^{-\beta}\, \textnormal{d}x + \int_{\{ \|x\| > 1 \}} p(x)\, \textnormal{d}x \bigg)\\
& \quad \leq |\tau|^{-\beta/\alpha} \bigg( |\widetilde{\tau}|^{\, d - \beta} + \int_{|\widetilde{\tau}|}^1 \int_{S^{d-1}} \|ry\|^{-\beta} \cdot r^{d-1}\, \textnormal{d}y\, \textnormal{d}r +1 \bigg)\\
& \quad \lesssim |\tau|^{-\beta/\alpha} \bigg( |\widetilde{\tau}|^{\, d - \beta} +  \int_{|\widetilde{\tau}|}^1 r^{d-\beta-1}\, \textnormal{d}r \bigg)\\
& \quad \lesssim |\tau|^{-\beta/\alpha} \cdot |\widetilde{\tau}|^{\, d - \beta} = |\tau|^{-\beta/\alpha} \cdot |\tau|^{(1-1/\alpha)(d - \beta)} = |\tau|^{(1-1/\alpha) d - \beta}.
\end{align*}
This proves (\ref{eq:Kernel_tau}). To prove (\ref{eq:Kernel_delta}) consider the region $\|x\| \leq |\widetilde{\delta}|/2$ which yields
\begin{equation*}
\| \mathrm{ sign }(\tau) \cdot x + \widetilde{\delta}\| \geq \big\||x\| - \|\widetilde{\delta}\| \big|= \|\widetilde{\delta}\| - \|x\| \geq \frac{1}{2} \cdot \|\widetilde{\delta}\|.
\end{equation*}
Thus for the estimates in (\ref{eq:Kernel_delta}) we have
\begin{align*}
& \E\big[\|(\tau, \mathrm{ sign }(\tau) \cdot X_{|\tau|}+\delta)\|^{-\beta}\big] = |\tau|^{-\beta/\alpha} \int_\Rd \|(\widetilde{\tau},  \mathrm{ sign }(\tau) \cdot x+ \widetilde{\delta})\|^{-\beta} \cdot p(x)\, \textnormal{d}x\\
& \quad \lesssim \|\delta\|^{-\beta} + \underbrace{|\tau|^{-\beta/\alpha} \int_{\{\|x\|\geq \|\widetilde{\delta}\|/2, \, \| \mathrm{ sign }(\tau) \cdot x+\widetilde{\delta}\| \geq |\widetilde{\tau}|\}} \| \mathrm{ sign }(\tau) \cdot x+\widetilde{\delta}\|^{-\beta} \cdot p(x) \, \textnormal{d}x}_{=:I_{1}}\\
&\quad \quad +\ \underbrace{|\tau|^{-\beta/\alpha} \int_{\{\|x\|\geq \|\widetilde{\delta}\|/2, \, \| \mathrm{ sign }(\tau) \cdot x+\widetilde{\delta}\| \leq |\widetilde{\tau}|\}} \widetilde{\tau}^{-\beta} \cdot p(x) \, \textnormal{d}x}_{=:I_{2}}.
\end{align*}

Now,
\begin{align*}
I_{1} &= |\tau|^{-\beta/\alpha} \int_{\{\|x\|\geq \|\widetilde{\delta}\|/2, \, \| \mathrm{ sign }(\tau) \cdot x+\widetilde{\delta}\| \geq |\widetilde{\tau}|\}} \| \mathrm{ sign }(\tau) \cdot x+\widetilde{\delta}\|^{-\beta} \cdot p(x) \, \textnormal{d}x\\
&= |\tau|^{-\beta/\alpha} \int_{\{\|x\|\geq \|\widetilde{\delta}\|/2, \, \|\mathrm{ sign }(\tau) \cdot x+\widetilde{\delta}\| \geq |\widetilde{\tau}|, \, \| \mathrm{ sign }(\tau) \cdot x+\widetilde{\delta}\| \geq \|\widetilde{\delta}\|\}} \|  \mathrm{ sign }(\tau) \cdot x+\widetilde{\delta}\|^{-\beta} \cdot p(x) \, \textnormal{d}x\\
&\quad +\ |\tau|^{-\beta/\alpha} \int_{\{\|x\|\geq \|\widetilde{\delta}\|/2, \, \| \mathrm{ sign }(\tau) \cdot x+\widetilde{\delta}\| \geq |\widetilde{\tau}|, \, \| \mathrm{ sign }(\tau) \cdot x+\widetilde{\delta}\| \leq \|\widetilde{\delta}\|\}} \| \mathrm{ sign }(\tau) \cdot x+\widetilde{\delta}\|^{-\beta} \cdot p(x) \, \textnormal{d}x\\
&\leq \|\delta\|^{-\beta} + \underbrace{|\tau|^{-\beta/\alpha} \int_{\{\|x\|\geq \|\widetilde{\delta}\|/2, \, \|\widetilde{\delta}\| \geq \| \mathrm{ sign }(\tau) \cdot x+\widetilde{\delta}\| \geq |\widetilde{\tau}|\}} \| \mathrm{ sign }(\tau) \cdot x+\widetilde{\delta}\|^{-\beta} \cdot p(x) \, \textnormal{d}x}_{=: I_{3}}.
\end{align*}

By using (\ref{eq:p_tail_estimates}) we further have
\begin{equation}\label{kernel-estimste:2}\begin{split}
I_{3} &= |\tau|^{-\beta/\alpha} \int_{\{\|x\|\geq \|\widetilde{\delta}\|/2, \, \|\widetilde{\delta}\| \geq \| \mathrm{ sign }(\tau) \cdot x+\widetilde{\delta}\| \geq |\widetilde{\tau}|\}} \|\mathrm{ sign }(\tau) \cdot x +\widetilde{\delta}\|^{-\beta} \cdot p(x) \, \textnormal{d}x\\
&\lesssim |\tau|^{-\beta/\alpha} \cdot \|\widetilde{\delta}\|^{-d-\alpha} \int_{\{\|\widetilde{\delta}\| \geq \| \mathrm{ sign }(\tau) \cdot x+\widetilde{\delta}\| \geq |\widetilde{\tau}|\}} \| \mathrm{ sign }(\tau) \cdot x+\widetilde{\delta}\|^{-\beta}\, \textnormal{d}x\\
&= |\tau|^{-\beta/\alpha} \cdot \|\widetilde{\delta}\|^{-d-\alpha} \int_{|\widetilde{\tau}|}^{\|\widetilde{\delta}\|} r^{d-\beta-1}\, \textnormal{d}r.
\end{split}\end{equation}

For $\alpha \in [1,2)$, $\beta < d$ and $|\tau| \leq \|\delta\|^\alpha$ by \eqref{kernel-estimste:2} we get
\begin{align*}
I_{3} &\lesssim |\tau|^{-\beta/\alpha}\cdot \|\widetilde{\delta}\|^{-\alpha-\beta}= |\tau|\cdot \|\delta\|^{-\alpha-\beta}\lesssim \|\delta\|^{-\beta},
\end{align*}

whereas for $\alpha \in (0,1]$ and $|\tau|\leq \|\delta\|$ one has
\begin{align*}
I_{3} \lesssim |\tau| \cdot \|\delta\|^{-\alpha-\beta} \leq \|\delta\|^{1-\alpha} \cdot \|\delta\|^{-\beta} \lesssim \|\delta\|^{-\beta}.
\end{align*}
For $\alpha \in [1,2)$, $\beta > d$ and $|\tau| \leq \|\delta\|^\alpha$ by \eqref{kernel-estimste:2} one has
\begin{align*}
I_{3} & \lesssim |\tau|^{-\beta/\alpha} \cdot \|\widetilde{\delta}\|^{-d - \alpha} \int_{\widetilde{\tau}}^\infty r^{-\beta} \cdot r^{d-1}\, \textnormal{d}r \lesssim |\tau|^{-\beta/\alpha} \cdot \|\widetilde{\delta}\|^{-d - \alpha} \cdot \widetilde{\tau}^{\, d - \beta}\\
& = |\tau|^{d+1-\beta} \cdot \|\delta\|^{-d-\alpha}  \leq \|\delta\|^{\alpha(d+1-\beta)} \cdot \|\delta\|^{-d-\alpha} = \|\delta\|^{(\alpha-1) d - \alpha\beta}.
\end{align*}

Finally, by using (\ref{eq:p_tail_estimates}) we get
\begin{align*}
I_{2} & = |\tau|^{-\beta/\alpha}  \int_{\{\|x\| \geq \|\widetilde{\delta}\|/2,\, \| \mathrm{ sign }(\tau) \cdot x + \widetilde{\delta}\| \leq |\widetilde{\tau}|\}} |\widetilde{\tau}|^{-\beta} \cdot p(x)\, \textnormal{d}x\\
& \lesssim |\tau|^{-\beta/\alpha} \cdot |\widetilde{\tau}|^{-\beta} \int_{\| x + \mathrm{ sign }(\tau) \cdot \widetilde{\delta}\| \leq |\widetilde{\tau}|\}} \|x\|^{-d-\alpha} \, \textnormal{d}x \\
& \lesssim  |\tau|^{-\beta/\alpha} \cdot |\widetilde{\tau}|^{-\beta} \cdot \|\widetilde{\delta}\|^{-d-\alpha} \cdot |\widetilde{\tau}|^d\\
& = |\tau|^{d-\beta + 1} \cdot \|\delta\|^{-d-\alpha}
\end{align*}
using the volume of a ball with radius $|\widetilde{\tau}|$ and center $- \mathrm{ sign }(\tau) \cdot \widetilde{\delta}$. Now, $\alpha \in (0,1]$, $\beta < d$ and $|\tau| \leq \|\delta\|$ result in
\begin{align*}
I_{2} & \lesssim |\tau|^{d-\beta + 1} \cdot \|\delta\|^{-d-\alpha} \leq \|\delta\|^{1 - \alpha - \beta} \leq \|\delta\|^{-\beta}.
\end{align*}
If $\alpha \in [1,2)$, $\beta \leq d$ and $|\tau| \leq \|\delta\|^\alpha$ one has
\begin{align*}
I_{2} & \lesssim |\tau|^{d-\beta + 1} \cdot \|\delta\|^{-d-\alpha} \leq \|\delta\|^{(\alpha-1) \cdot d - \alpha\beta} \leq \|\delta\|^{-\beta}.
\end{align*}
If $\alpha \in [1,2)$, $\beta \geq d$ and $|\tau| \leq \|\delta\|^\alpha$ one has
\begin{align*}
I_{2} & \lesssim \|\delta\|^{(\alpha-1) \cdot d - \alpha\beta}\quad\text{ and }\quad \|\delta\|^{-\beta}\leq \|\delta\|^{(\alpha-1) \cdot d - \alpha\beta}.
\end{align*}
Altogether, we have shown
\begin{equation*}
\E\big[\|(\tau, \mathrm{ sign }(\tau) \cdot X_{|\tau|}+\delta)\|^{-\beta}\big] \lesssim \|\delta\|^{-\beta} + I_{1} + I_{2} \lesssim \|\delta\|^{-\beta} + I_{3}+I_{2}
\end{equation*}
and our upper bounds for $I_{2}$ and $I_{3}$ directly yield \eqref{eq:Kernel_delta}.
\end{proof}

Now, we are able to calculate the lower bound via energy estimates.

\begin{theorem}\label{le:EnergyEstimates}
Let $T\subseteq \R_+$ be a Borel set and $\alpha \in (0,2)$. Let $X=(X_t)_{t\geq0}$ be an isotropic $\alpha$-stable L\'{e}vy process in $\Rd$ and $f:T\rightarrow\{y\in\Rd:\|y-x\|\leq\frac12\}$ for fixed $x \in \Rd$ be a Borel measurable function.  Let $\ph_\alpha = \mathcal{P}^\alpha\m\dim \pazocal{G}_T(f)$. Then one $\Pb$-almost surely has
\begin{equation}\label{eq:lower_bound}
\dim \pazocal{G}_T(X+f) \geq
\begin{cases}
\ph_1, & \alpha \in (0,1],\\
\ph_\alpha \ \wedge \ \left(\frac{1}{\alpha} \cdot \ph_\alpha + \big(1-\frac{1}{\alpha}\big)\cdot d\right), & \alpha \in [1,2].
\end{cases}
\end{equation}
\end{theorem}

\begin{proof}
We define the increments $\tau := t-s$ and $\delta := f(t)-f(s)$ with $\|\delta\| \in [0,1]$ and consider the difference kernel $K^\beta(t,x) = \E \big[\|(t, \mathrm{ sign }(t) \cdot X_{|t|} + x)\|^{-\beta}\big]$. We prove that $\pazocal{E}_{K^\beta}(\mu)<\infty$ holds for $\mu \in \pazocal{M}^1(\pazocal{G}_T(f))$ from the parabolic version of Frostman's lemma in Theorem \ref{th:FrostmansLemma} and for every $\beta$ less than the right-hand side of \eqref{eq:lower_bound}.
Then the claim follows due to Lemma \ref{le:converted_kernel}.
For the energy integral we have
\begin{equation}\label{energy-integral}\begin{split}
& \pazocal{E}_{K^\beta}(\mu) = \int \int_{\pazocal{G}_T(f)\times \pazocal{G}_T(f)}  K^\beta(t-s,f(t)-f(s))\, \textnormal{d}\mu (s,x)\, \textnormal{d}\mu (t,y)\\
& \quad \leq \int \int_{\{|t-s| \in (0,1] \}} K^\beta(\tau,\delta)\, \textnormal{d}\mu \, \textnormal{d}\mu + \int \int_{\{|t-s| \in (1, \infty) \}} |t-s|^{-\beta}\, \textnormal{d}\mu \, \textnormal{d}\mu\\
& \quad \lesssim \int \int_{\{|\tau| \in (0,1]\}} \E \big[\|(\tau, \mathrm{ sign }(\tau) \cdot X_{|\tau|}(\omega) + \delta)\|^{-\beta}\big]\, \textnormal{d}\mu \, \textnormal{d}\mu.
\end{split}\end{equation}

\emph{(i)}
We begin with the case $\alpha \in (0,1]$ and $\beta = \ph_1 -2 \eps $ for some arbitrary $\eps > 0$.  Due to Lemma \ref{le:kernel_estimate} we have
\begin{align*}
\pazocal{E}_{K^\beta}(\mu) 
\lesssim&\ \underbrace{\int \int_{\{|\tau| \in (0, 1], \,  \|\delta\| \in [0,\, |\tau| ] \}} |\tau|^{-\beta}\, \textnormal{d}\mu \, \textnormal{d}\mu}_{=:\ I_1} + \underbrace{\int \int_{\{|\tau| \in (0,\, 1], \,  \|\delta\| \in ( |\tau|,\, 1]\}} \|\delta\|^{-\beta}\, \textnormal{d}\mu \, \textnormal{d}\mu}_{=:\ I_2}.
\end{align*}
We get
\begin{align*}
I_1
\lesssim&\ \sum_{k=1}^\infty  2^{k\beta} \cdot \mu\otimes \mu\,\big{\{ |\tau| \in \big(2^{-k},\, 2 \cdot 2^{-k}\big], \,  \|\delta\| \in \big[0,2 \cdot 2^{-k}\big]\big\}}.
\end{align*}
Further, 
\begin{align*}
I_2 \lesssim&\ \sum_{k=1}^\infty 2^{k\beta} \cdot \mu\otimes \mu \, \big{\{ |\tau| \in \big(0,2\cdot 2^{-k}\big],\  \|\delta\| \in \big(2^{-k},  \, 2\cdot 2^{-k}\big]\big\}}.
\end{align*}

Now we have to calculate the expressions $\mu \otimes \mu \{\cdot\}$ for $I_1$ and $I_2$. For each $k\in\N$ we tile $\R_{+}\times\Rd$ by disjoint hypercubes of size $2^{-k} \times \cdots \times 2^{-k}$ and denote the collection of such hypercubes by $\pazocal{D}_k$. For every $c \in (0,1]$, $\gamma  = \ph_1 - \eps$ and $\alpha \in (0,1]$ Frostman's lemma \ref{th:FrostmansLemma} yields
\begin{equation*}
\mu\Bigg([t,\, t+c]\times \prod_{i=1}^d\, \big[x_i,\, x_i+c\big]\Bigg) \lesssim 
c^{\gamma},
\end{equation*}
in particular we have $\mu(Q')\lesssim 2^{-k\gamma}$ for each $Q'\in\pazocal{D}_k$. In order to estimate $I_1$ we define the following relation on $\pazocal{D}_k$. For two hypercubes $Q,Q'\in\pazocal{D}_k$ we write $Q\sim Q'$ if there exists $(s,x)\in Q$ and $(t,y) \in Q'$ such that $|\tau|=|t-s| \in \big(2^{-k},\, 2 \cdot 2^{-k}\big]$ and $\|\delta\|=\|y-x\| \in \big[0,\, 2 \cdot 2^{-k}\big]$. Thus
\begin{align*}
I_1 \lesssim &\ \sum_{k=1}^\infty  2^{k\beta} \sum_{\substack{Q,Q' \in \pazocal{D}_k \\ Q \sim Q'}} \mu\otimes \mu\, (Q \times Q') = \sum_{k=1}^\infty  2^{k\beta} \sum_{\substack{Q,Q' \in \pazocal{D}_k \\ Q \sim Q'}} \mu(Q)\cdot \mu(Q').
\end{align*}
The number of hypercubes related to a fixed $Q\in\pazocal{D}_{k}$ via $\sim$ is bounded by a universal constant not depending on $k$ and $Q$, hence
\begin{align*}
I_1 & \lesssim \sum_{k=1}^\infty  2^{k\beta} \sum_{Q\in\pazocal{D}_k} \sum_{Q \sim Q'} \mu(Q)\cdot \mu(Q')\lesssim \sum_{k=1}^\infty  2^{k(\beta-\gamma)} \cdot \sum_{Q\in\pazocal{D}_k} \mu(Q).
\end{align*}
Note that $\sum_{Q\in \pazocal{D}_k}\mu(Q)=\mu\big(\bigcup_{Q\in \pazocal{D}_k}Q\big) =\mu\big(\R_{+}\times\Rd\big) =1$ and we conclude
\begin{align*}
I_1 \lesssim&\ \sum_{k=1}^\infty  2^{k(\beta-\gamma)} = \sum_{k=1}^\infty 2^{-k\eps} < \infty,
\end{align*} 
since $\beta = \ph_1 -2 \eps$ and $\gamma = \ph_1 - \eps$. For the estimation of $I_2$ we define another relation on $\pazocal{D}_k$. For two hypercubes $Q,Q'\in\pazocal{D}_k$ we write $Q \approx Q'$ if there exists $(s,x)\in Q$ and $(t,y) \in Q'$ such that $|\tau| \in \big(0,\, 2 \cdot 2^{-k}\big]$ and $\|\delta\| \in \big(2^{-k},  \, 2 \cdot 2^{-k}\big]$. Thus
\begin{align*}
I_2 \lesssim \sum_{k=1}^\infty  2^{k\beta} \sum_{\substack{Q,Q' \in \pazocal{D}_k \\ Q \approx Q'}} \mu(Q)\cdot \mu(Q').
\end{align*}
Again, the number of hypercubes related to some fixed $Q$ via $\approx$ is bounded by a universal constant.  Hence the same calculation as for $I_1$ yields
\begin{align*}
I_2 \lesssim \sum_{k=1}^\infty  2^{k(\beta-\gamma)} = \sum_{k=1}^\infty 2^{-k\eps} < \infty.\\
\end{align*} 

\emph{(ii)} Now we treat the case $\alpha \in [1,2)$ and $\ph_{\alpha}\leq d$. Let $\beta = \ph_\alpha - 2 \alpha \cdot \eps < d$ for some arbitrary $\eps > 0$. Due to Lemma \ref{le:kernel_estimate} we have
\begin{equation*}
\pazocal{E}_{K^\beta}(\mu) 
\lesssim  \underbrace{\int \int_{\{|\tau| \in (0,\, 1], \,  \|\delta\| \in [0,\, |\tau|^{1/\alpha} ] \}} |\tau|^{-\beta/\alpha}\, \textnormal{d}\mu \, \textnormal{d}\mu}_{=:\ I_3}
 + \underbrace{\int \int_{\{|\tau| \in (0,\, 1], \,  \|\delta\| \in ( |\tau|^{1/\alpha},\, 1]\}} \|\delta\|^{- \beta}\, \textnormal{d}\mu \, \textnormal{d}\mu}_{=:\ I_4}.
\end{equation*}
We get
\begin{align*}
I_3
\lesssim&\ \sum_{k=1}^\infty  2^{k\beta/\alpha} \cdot \mu\otimes \mu\, \big\{ |\tau| \in \big(2^{-k},\, 2 \cdot 2^{-k}\big], \,  \|\delta\| \in \big[0,2^{1/\alpha} \cdot 2^{-k / \alpha}\big]\big\}.
\end{align*}
Further, 
\begin{align*}
I_4 \lesssim&\ \sum_{k=1}^\infty 2^{k\beta/\alpha}\cdot \mu\otimes \mu\, \big\{ |\tau| \in \big(0,\, 2 \cdot 2^{-k}\big],\  \|\delta\| \in \big(2^{-k/\alpha},  \, 2^{1/\alpha} \cdot 2^{-k/\alpha}\big]\big\}.
\end{align*}

Now we have to calculate the expressions $\mu \otimes \mu \{\cdot\}$ for $I_3$ and $I_4$. For each $k\in\N$ we tile $\R_{+}\times\Rd$ by disjoint $\alpha$-parabolic cylinders of size $2^{-k} \times  2^{-k/\alpha} \times \cdots \times 2^{-k/\alpha}$ and again denote the collection of such cylinders by $\pazocal{D}_k$. For every $c \in (0,1]$, $\gamma = \ph_\alpha - \alpha \cdot \eps$ and $\alpha \in [1,2)$ Frostman's lemma \ref{th:FrostmansLemma} yields
\begin{equation*}\label{eq:Frostmans_Lemma_alpha_0,1}
\mu\Bigg([t,\, t+c]\times \prod_{i=1}^d\, \big[x_i,\, x_i+c^{1/\alpha}\big]\Bigg) \lesssim 
c^{\gamma/\alpha},
\end{equation*} 
in particular we have $\mu(Q')\lesssim 2^{-k\gamma/\alpha}$ for each $Q'\in\pazocal{D}_k$.
The same technique as in \textit{(i)} results in
\begin{align*}
I_3,I_4 \lesssim&\ \sum_{k=1}^\infty  2^{k(\beta - \gamma)/\alpha} \leq \sum_{k=1}^\infty 2^{-k\eps} < \infty,
\end{align*} 
since $\beta = \ph_\alpha - 2\alpha \cdot \eps$ and $\gamma = \ph_\alpha - \alpha \cdot \eps$.\\

\emph{(iii)} Finally, we treat the case $\alpha \in [1,2)$ and $\ph_{\alpha}>d$. Let $\beta = (1-\frac1{\alpha}) \cdot d + \frac1{\alpha}\cdot\ph_\alpha - 2\eps > d$ for for sufficiently small $\eps > 0$.  Due to Lemma \ref{le:kernel_estimate} we have

\begin{align*}
\pazocal{E}_{K^\beta}(\mu) 
& \lesssim \int \int_{\{|\tau| \in (0,\, 1], \, \|\delta\| \in [0,\, |\tau|^{1/\alpha} ] \}} |\tau|^{(1-1/\alpha) d - \beta},\, \textnormal{d}\mu \, \textnormal{d}\mu\\
& \quad + \int \int_{\{|\tau| \in (0,\, 1], \, \|\delta\| \in ( |\tau|^{1/\alpha},\, 1]\}} \|\delta\|^{(\alpha-1) d-\alpha\beta}\, \textnormal{d}\mu \, \textnormal{d}\mu
\end{align*}
and the same techniques as in \emph{(ii)} yield the finiteness of this expression when choosing $\gamma=\ph_{\alpha}-\varepsilon$ in Frostman's lemma \ref{th:FrostmansLemma}.
\end{proof}

\section{Range: Upper and Lower Bounds}

We give upper and lower bounds for the Hausdorff dimension of the range of a stable L\'{e}vy process with drift.

\begin{theorem}\label{thm:Range_dim_X+f_leq_dim_f}
Let $T\subseteq \R_+$ be any set and $\alpha \in (0,2]$. Let $X=(X_t)_{t\geq0}$ be an isotropic $\alpha$-stable L\'{e}vy process in $\Rd$ and $f:T\rightarrow \Rd$ be any function. Define $\ph_\alpha := \mathcal{P}^\alpha\m\dim \pazocal{G}_T(f)$.  Then one $\Pb$-almost surely has
\begin{equation}
\dim \pazocal{R}_T(X+f) \leq
\begin{cases}
\left(\alpha \cdot \ph_\alpha\right) \ \wedge \ d, & \alpha \in (0,1],\\
\ph_\alpha \ \wedge \ d , & \alpha \in [1,2].
\end{cases}
\end{equation}
\end{theorem}

\begin{proof}
The Gaussian case follows from the proof of Theorem 1.2 in \cite{PS16} and Proposition \ref{pro:P2=2P}. 
Since the Hausdorff dimension of the range never exceeds the topological dimension of the space a function maps to we always have $\dim \pazocal{R}_T(X+f) \leq d$.
In case of $\alpha \in [1,2)$ the claim directly follows from Theorem \ref{le:General_Upper_Bound_Hausdorff} and Theorem \ref{thm:dim_X+f_leq_dim_f} which yield
\begin{equation*}
\dim \pazocal{R}_T(X+f) \leq \dim \pazocal{G}_T(X+f) \leq \mathcal{P}^\alpha\m\dim \pazocal{G}_T(X+f) \leq \mathcal{P}^\alpha\m\dim \pazocal{G}_T(f)=\ph_{\alpha}.
\end{equation*}

Now, let $\alpha \in (0,1]$, $\beta = \alpha \cdot \ph_\alpha$ and let $\delta,\, \eps>0$ be arbitrary. Then $\pazocal{G}_T(f)$ can be covered by $\alpha$-parabolic cylinders
\begin{equation*}
(\mathsf{P}_{k})_{k \in \N} = \Bigg([t_k,\, t_k + c_k ] \times \prod_{i=1}^d\, \big[x_{i,k},\, x_{i,k}+c_k^{1/\alpha} \big] \Bigg)_{k\in\N} \subseteq \mathcal{P}^\alpha 
\end{equation*}
such that $\sum_{k=1}^\infty |\mathsf{P}_{k}|^{(\beta + \delta)/\alpha} \lesssim \sum_{k=1}^\infty c_k^{(\beta + \delta) /\alpha} \leq \eps$. Let $M_k(\omega)$ be the random number of a fixed $2^d$-nested collection of hypercubes with sidelength $c_k^{1/\alpha}$ that the path $t\mapsto X_t(\omega)$ hits at some time $t\in[t_k,\, t_k+c_k]$. As in the proof of Theorem \ref{thm:dim_X+f_leq_dim_f} we obtain a random parabolic cover $\bigcup_{k\in\N}\ \widetilde{\mathsf{P}}_{k}(\omega) \supseteq \pazocal{G}_T(X(\omega)+f)$ where
\begin{align*}
\widetilde{\mathsf{P}}_{k}(\omega) = [t_k,\, t_k + c_k] \times \bigcup_{j=1}^{M_k(\omega)} \prod_{i=1}^d & \left( \big[\xi_{i,j,k}(\omega)+x_{i,k},\, \xi_{i,j,k}(\omega)+x_{i,k}+c_k^{1/\alpha}\big]\right.\\
\cup\ & \left. \big[\xi_{i,j,k}(\omega)+x_{i,k}+c_k^{1/\alpha},\, \xi_{i,j,k}(\omega)+x_{i,k}+2c_k^{1/\alpha}\big]\right).
\end{align*}
By projection we get the random cover
$\bigcup_{k\in\N} \square_{k} \supseteq \pazocal{R}_T(X(\omega) + f)$ with
\begin{align*}
\square_{k}(\omega) = \bigcup_{k=1}^{M_k(\omega)} \prod_{i=1}^d & \left(\big[\xi_{i,j,k}(\omega)+x_{i,k},\xi_{i,j,k}(\omega)+x_{i,k}+c_k^{1/\alpha}\big]\right.\\
\cup\ &\left.\big[\xi_{i,j,k}(\omega)+x_{i,k}+c_k^{1/\alpha},\xi_{i,j,k}(\omega)+x_{i,k}+2c_k^{1/\alpha}\big]\right).
\end{align*}
This is a union of $M_k(\omega)\cdot 2^d$ hypercubes with diameter $|\square_{k}(\omega)|\asymp c_k^{1/\alpha}$. 
As in the proof of Theorem \ref{thm:dim_X+f_leq_dim_f} we get $\E[M_k] \lesssim c_k^{-\delta'/\alpha}$.
Hence we get for $\eps' = \delta + \delta' > 0$
\begin{align*}
& \E\Big[\pazocal{H}^{\beta+\eps'}(\pazocal{R}_T(X+f))\Big]
\leq \E\bigg[\sum_{k=1}^\infty \big|\square_{k}\big|^{\beta+\eps'} \bigg]
\lesssim \E\bigg[\sum_{k=1}^\infty M_k\cdot 2^d \cdot c_k^{(\beta+\eps')/\alpha}\bigg]\\
& \quad \lesssim \sum_{k=1}^\infty \E[M_k]\cdot c_k^{(\beta+\eps')/\alpha}
\lesssim \sum_{k=1}^\infty c_k^{(\beta+\eps'- \delta')/\alpha}
= \sum_{k=1}^\infty  c_k^{(\beta + \delta)/\alpha}
\leq \eps.
\end{align*}
Since $\eps, \eps' > 0$ are arbitrary, for all $\beta'>\beta$ we get
$\E\Big[\pazocal{H}^{\beta'}(\pazocal{R}_T(X+f))\Big] =0$,
hence
\begin{equation*}
\pazocal{H}^{\beta'}(\pazocal{R}_T(X+f)) =0\quad \Pb\text{-almost surely}.
\end{equation*}
Since $\beta'>\beta$ is also arbitrary, we finally get $\dim \pazocal{R}_T(X+f) \leq \beta = \alpha \cdot \ph_\alpha$ $\Pb$-almost surely, as claimed.
\end{proof}

The lower bound is obtained by the energy method. The stable process $X$ is transformed into the kernel of the energy integral.

\begin{lemma}\label{le:converted_kernel_2}
Let $T \subseteq\R_+$ be a Borel set and $\alpha\in (0,2)$. Let $X=(X_t)_{t\geq0}$ be an isotropic stable L\'{e}vy process in $\Rd$ and $f:\R_+\rightarrow\Rd$ be a Borel measurable function.  Define the difference kernel 
\begin{equation*}
\kappa^\beta(t,x) := \E \big[\|\mathrm{sign}(t) \cdot X_{|t|} + x\|^{-\beta}\big].
\end{equation*}
Then $\Ca_{\kappa^\beta}(\pazocal{G}_T(f)) > 0$ implies that $\Pb$-almost surely $\Ca_\beta(\pazocal{R}_T(X(\omega)+f)) > 0$ holds.
Hence $\pazocal{E}_{\kappa^\beta}(\mu)<\infty$ for some probability measure $\mu\in\pazocal{M}^1(\pazocal{G}_T(f))$ implies $\dim \pazocal{R}_T(X+f) \geq \beta$ $\,\Pb$-almost surely.
\end{lemma}

\begin{proof}
Let $\mu \in \pazocal{M}^1(\pazocal{G}_T(f))$ and $\pi_t$ denote the projection onto the time component, i.e. $\pi_t(t,f(t)) = t$. Define the probability measure $\nu \in \pazocal{M}^1(\R_+)$ as the pushforward measure $\nu(A) = \mu\big(\pi_t^{-1}(A)\big)$ for Borel sets $A\subseteq \R_{+}$ and further the random probability measure $\widetilde{\mu}_\omega(R)= \nu((X(\omega)+f)^{-1}(R))$ for every Borel set $R \subseteq \Rd$. Then Tonelli's theorem and the stationarity of the increments of $X$ yield
\begin{align*}
\E\big[\pazocal{E}_\beta (\widetilde{\mu}_\omega)\big] &
= \E\bigg[\int_{\pazocal{R}_T(X(\omega)+f)}  \int_{\pazocal{R}_T(X(\omega)+f)} \|x-y\|^{-\beta}\, \textnormal{d}\widetilde{\mu}_\omega(x)\, \textnormal{d}\widetilde{\mu}_\omega(y)\bigg]\\
&= \E\bigg[\int_{T} \int_{T} \|X_t(\omega) + f(t)-(X_s(\omega) + f(s))\|^{-\beta}\, \textnormal{d}\nu(t)\, \textnormal{d}\nu(s)\bigg]\\
&= \int_{\pazocal{G}_T(f)} \int_{\pazocal{G}_T(f)} \E\big[\|X_t(\omega) - X_s(\omega) +x- y\|^{-\beta}\big]\, \textnormal{d}\mu(t,x)\, \textnormal{d}\mu(s,y)\\
&= \int_{\pazocal{G}_T(f)} \int_{\pazocal{G}_T(f)} \E\big[\|\mathrm{ sign}(t-s) \cdot X_{|t-s|}(\omega) + x - y\|^{-\beta}\big]\, \textnormal{d}\mu(t,x)\, \textnormal{d}\mu(s,y)\\
&= \pazocal{E}_{\kappa^\beta} (\mu).
\end{align*}

By assumption, there exists $\mu \in \pazocal{M}^1(\pazocal{G}_T(f))$ such that $\pazocal{E}_{\kappa^\beta} (\mu)< \infty$ 
holds.  From that one $\Pb$-almost surely has $\pazocal{E}_\beta (\widetilde{\mu}_\omega) < \infty$ and the final statement immediately follows by Frostman's Theorem \ref{th:Frostman's_Theorem} since the range of a Borel set under a Borel measurable function is a Suslin set, see Section 11 in \cite{J03}.
\end{proof}

Next we provide estimates for the difference kernel $\kappa^\beta(\tau,\delta) = \E \big[\|\mathrm{sign}(t) \cdot X_{|t|}+ x\|^{-\beta}\big]$ from Lemma \ref{le:converted_kernel_2} that will give appropriate estimates of the energy integral.

\begin{lemma}\label{le:kernel_estimate_2}
Let $\alpha \in (0,2)$ and $X=(X_t)_{t \geq 0}$ be an isotropic $\alpha$-stable L\'{e}vy process in $\Rd$.  Let $\beta \in (0,d)$ and $\tau\in\R$, $\delta\in\Rd$ be such that $|\tau|\in(0,1]$, $\|\delta\| \in [0,1]$.  Then one has
\begin{equation*}
\kappa^\beta(\tau,\delta) \lesssim \begin{cases}
|\tau|^{-\beta/\alpha}\\
\|\delta\|^{-\beta} & \text{ for }|\tau| \leq \|\delta\|^\alpha.
\end{cases}
\end{equation*}
\end{lemma}

\begin{proof}
Let $p(x)$ denote the density function of $X_1 \eqd |\tau|^{-1/\alpha} X_{|\tau|}$. We define the rescaled increment $\widetilde{\delta} := \delta/|\tau|^{1/\alpha}$. The self-similarity of the stable L\'{e}vy process and Lemma \ref{le:density_monotocity} yield
\begin{align*} 
& \E\big[\|\mathrm{sign}(\tau) \cdot X_{|\tau|}+\delta\|^{-\beta}\big] = |\tau|^{-\beta/\alpha} \int_\Rd \|\mathrm{sign}(\tau) \cdot x+\widetilde{\delta}\|^{-\beta} \cdot p(x)\, \textnormal{d}x \\
& \quad= |\tau|^{-\beta/\alpha} \int_\Rd \|x+\mathrm{sign}(\tau) \cdot \widetilde{\delta}\|^{-\beta} \cdot p(x)\, \textnormal{d}x \\
& \quad \lesssim |\tau|^{-\beta/\alpha} \int_\Rd \|x\|^{-\beta} \cdot p(x) \, \textnormal{d}x \lesssim |\tau|^{-\beta / \alpha} \cdot \E\big[\|X_1\|^{-\beta}\big]
\lesssim |\tau|^{-\beta / \alpha},
\end{align*}
since negative moments of order $\beta<d$ exist; see Lemma 3.1 in \cite{BMS}. Now consider the region $\|\mathrm{sign}(\tau) \cdot x+\widetilde{\delta}\| \leq \|\widetilde{\delta}\|/2$ which yields
\begin{equation*}
\|x\| = \|\mathrm{sign}(\tau) \cdot x + \widetilde{\delta} - \widetilde{\delta}\| \geq \big\|| \mathrm{sign}(\tau) \cdot x + \widetilde{\delta}\| - \|\widetilde{\delta}\| \big|= \|\widetilde{\delta}\| - \|\mathrm{sign}(\tau) \cdot x + \widetilde{\delta}\| \geq \frac{1}{2} \cdot \|\widetilde{\delta}\|.
\end{equation*}
Thus $\beta < d$ and $\tau \leq \|\delta\|^\alpha$ lead to
\begin{align*}
& \E\big[\|\mathrm{sign}(\tau) \cdot X_{|\tau|} + \delta\|^{-\beta}\big] = |\tau|^{-\beta/\alpha} \int_\Rd \|\mathrm{sign}(\tau) \cdot x + \widetilde{\delta}\|^{-\beta} \cdot p(x)\, \textnormal{d}x\\
& \quad \lesssim \|\delta\|^{-\beta} + |\tau|^{-\beta/\alpha} \int_{\{\|\mathrm{sign}(\tau) \cdot x + \widetilde{\delta}\|\leq \|\widetilde{\delta}\|/2\}} \|\mathrm{sign}(\tau) \cdot x+ \widetilde{\delta}\|^{-\beta} \cdot p(x)\, \textnormal{d}x \\
& \quad \lesssim \|\delta\|^{-\beta} + |\tau|^{-\beta/\alpha} \cdot \|\widetilde{\delta}\|^{-d-\alpha} \int_0^{\|\widetilde{\delta}\|} r^{d - \beta - 1}\, \textnormal{d}r = \|\delta\|^{-\beta} + |\tau|^{-\beta/\alpha} \cdot \|\widetilde{\delta}\|^{-\alpha-\beta} \\
& \quad = \|\delta\|^{-\beta} + |\tau| \cdot \|\delta\|^{-\alpha -\beta} \lesssim \|\delta\|^{-\beta},
\end{align*}
where we have used \eqref{eq:p_tail_estimates} to estimate the tail-densities.
\end{proof}

We get a lower bound for the Hausdorff dimension of the range of a stable L\'{e}vy process with drift.

\begin{theorem}\label{thm:Main_Theorem_Hausdorff_Range}
Let $T\subseteq \R_+$ be a Borel set and $\alpha \in (0,2)$. Let $X=(X_t)_{t\geq0}$ be an isotropic $\alpha$-stable L\'{e}vy process in $\Rd$ and $f:T\rightarrow \{y\in\Rd:\|y-x\|\leq\frac12\}$ for fixed $x \in \Rd$ be a Borel measurable function.  Let $\ph_\alpha = \mathcal{P}^\alpha\m\dim \pazocal{G}_T(f)$. Then one $\Pb$-almost surely has
\begin{equation}\label{eq:range_lower_bound}
\dim \pazocal{R}_T(X+f) \geq
\begin{cases}
\left(\alpha \cdot \ph_\alpha\right) \ \wedge \ d, & \alpha \in (0,1],\\
\ph_\alpha \ \wedge \ d, & \alpha \in [1,2).
\end{cases}
\end{equation}
\end{theorem}

\begin{proof}
We consider the difference kernel $\kappa^\beta(t,x) = \E \big[\|\mathrm{sign}(t) \cdot X_{|t|}(\omega) + x\|^{-\beta}\big]$ from Lemma \ref{le:converted_kernel_2}. Analogously to the proof of Theorem \ref{le:EnergyEstimates}, we can show that $\pazocal{E}_{\kappa^\beta}(\mu)$ is finite for $\mu \in \pazocal{M}^1(\pazocal{G}_T(f))$ from the parabolic version of Frostman's Lemma in Theorem \ref{th:FrostmansLemma} and for every $\beta$ less than the right-hand side of (\ref{eq:range_lower_bound}). Note that the different lower bounds for the dimension of the range in cases $\alpha \in (0,1]$ and $\alpha \in [1,2)$ occur due to the different upper bounds in the parabolic version of Frostman's Lemma.
\end{proof}

\section{Estimates for the Parabolic Hausdorff Dimension}

We first give an estimate for the $\alpha$-parabolic Hausdorff dimension of a constant function.

\begin{lemma}\label{thm:Nullfuntion}
Let $T\subseteq \R$ be any set and $\alpha \in (0,\infty)$. Define the constant function $f_{C}(x)= C \in \Rd$ for all $x\in T$. Then one has
\begin{equation*}
\mathcal{P}^\alpha\m\dim \pazocal{G}_T( f_{C}) \leq (\alpha \vee 1) \cdot \dim T
\end{equation*}
\end{lemma}

\begin{proof}
Without loss of generality, let $f_C = f_0 \equiv 0 \in \R^d$. 

\emph{(i)}  Let $\alpha \in (0,1]$, $\beta = \dim T$ and let $\delta,\eps >0$ be arbitrary. Then there exists a cover $\bigcup_{k\in\N}\ T_k \supseteq T$ with $T_k = [t_{k},\, t_{k}+c_k]$ and $c_{k}\leq1$ such that $
\sum_{k=1}^\infty |T_k|^{\beta+\delta} = \sum_{k=1}^\infty c_k^{\beta+\delta}\leq \eps.$ Now, $\pazocal{G}_T( f_{0})$ can be covered by $\alpha$-parabolic cylinders
\begin{equation*}
(\mathsf{P}_{k})_{k\in\N} = \Bigg([t_{k},\, t_{k}+c_k] \times \prod_{j=1}^d \ [0,c_k^{1/\alpha}]\Bigg)_{k\in\N} \subseteq \mathcal{P}^\alpha
\end{equation*}
with $|\mathsf{P}_{k}| \asymp c_k$. Hence 
\begin{align*}
\mathcal{P}^\alpha\m\pazocal{H}^{\beta + \delta}(\pazocal{G}_T(f_0)) \leq \sum_{k=1}^\infty |\mathsf{P}_{k}|^{\beta + \delta} \lesssim \sum_{k=1}^\infty c_k^{\beta + \delta} \leq \eps .
\end{align*}
Since $\delta>0$ is arbitrary,  for all $\beta' > \beta$ we have $\mathcal{P}^\alpha\m\pazocal{H}^{\beta'}(\pazocal{G}_T(f_0)) < \infty$ and therefore one has $\mathcal{P}^\alpha\m\dim \pazocal{G}_T(f_0)\leq \beta'$. Since $\beta' > \beta$ is also arbitrary, we obtain
\begin{gather*}
\mathcal{P}^\alpha\m\dim \pazocal{G}_T(f_0) \leq \beta = \dim T.
\end{gather*}

\emph{(ii)} Let $\alpha \in [1,\infty)$, $\beta = \alpha \cdot \dim T$ and let $\delta,\, \eps >0$ be arbitrary. 
With the cover $\bigcup_{k\in\N}\ T_k \supseteq T$ from part \emph{(i)} we get
$\sum_{k=1}^\infty |T_k|^{(\beta+\delta) / \alpha}=\sum_{k=1}^\infty c_k^{(\beta + \delta) / \alpha}\leq \eps$. 
Then the cover $\bigcup_{k\in\N} \mathsf{P}_{k}\supseteq \pazocal{G}_T( f_{0})$ from part \emph{(i)} now fulfills
$\big|\mathsf{P}_{k}\big| \asymp c_k^{1/\alpha}$ and it follows that
\begin{align*}
 \mathcal{P}^\alpha\m\pazocal{H}^{\beta + \delta}(\pazocal{G}_T(f_0))
\leq \sum_{k=1}^\infty \big|\mathsf{P}_{k}\big|^{\beta + \delta}
\lesssim \sum_{k=1}^\infty c_k^{(\beta + \delta)/\alpha}
\leq \eps.
\end{align*}
 Since $\delta>0$ and $\beta'>\beta$ are arbitrary, as in part \emph{(i)} we get $\mathcal{P}^\alpha\m\dim \pazocal{G}_T(f_0) \leq \beta = \alpha \cdot \dim T$.
\end{proof}

We can calculate the $\alpha$-parabolic Hausdorff dimension of the graph of an isotropic $\alpha$-stable L\'{e}vy process itself.

\begin{theorem}\label{thm:Dimension_X}
Let $T\subseteq \R_+$ be a Borel set and $\alpha \in (0,2]$. Let $X=(X_t)_{t\geq0}$ be an isotropic $\alpha$-stable L\'{e}vy process in $\Rd$. Then one $\Pb$-almost surely has
\begin{align*}
\mathcal{P}^\alpha\m\dim\pazocal{G}_T(X) = (\alpha \vee 1) \cdot \dim T.
\end{align*}
\end{theorem}

\begin{proof}
By Theorem 3.2 in \cite{W17}, Theorem \ref{le:General_Upper_Bound_Hausdorff}, Theorem \ref{thm:dim_X+f_leq_dim_f} and Lemma \ref{thm:Nullfuntion} in case $\alpha \cdot \dim T \geq 1$, i.e. $\alpha \in [1,2]$, and $f_0 \equiv 0 \in \R^d$ one $\Pb$-almost surely has
\begin{align*}
\dim T + 1-1/\alpha 
&=  \dim  \pazocal{G}_T(X) 
\leq 1/\alpha \cdot \mathcal{P}^\alpha\m\dim  \pazocal{G}_T(X) + 1-1/\alpha\\
&\leq 1/\alpha \cdot \mathcal{P}^\alpha\m\dim  \pazocal{G}_T(f_0) + 1-1/\alpha 
\leq \dim T + 1-1/\alpha.
\end{align*}
In the other cases, Theorem 3.1 in \cite{W17} together with the same theorems as above $\Pb$-almost surely yield 
\begin{align*}
& (\alpha \vee 1) \cdot \dim T = \dim  \pazocal{G}_T(X) \leq \mathcal{P}^\alpha\m\dim \pazocal{G}_T(X) \leq \mathcal{P}^\alpha\m\dim \pazocal{G}_T(f_0) \leq (\alpha \vee 1) \cdot \dim T
\end{align*}
and the claim follows.
\end{proof}

\begin{remark}
We can also deduce the Hausdorff dimension of the graph of the fractional Brownian motion $B^H = \big(B_t^H\big)_{t \geq 0}$ in $\Rd$ of Hurst index $1/\alpha = H \in (0,1]$. One $\Pb$-almost surely has
\begin{align*}
\mathcal{P}^\alpha\m\dim\pazocal{G}_{T}\big(B^H\big) = \frac{\dim T}{H} = \alpha \cdot \dim T.
\end{align*}
This follows from Theorem 2.1 in \cite{X95}, Proposition \ref{pro:P2=2P}, Theorem \ref{le:General_Upper_Bound_Hausdorff}, Lemma 2.2 in \cite{PS16} and Lemma \ref{thm:Nullfuntion} for $\alpha \cdot \dim T \leq d$ and $f_0 \equiv 0 \in \R^d$ which $\Pb$-almost surely yield
\begin{align*}
\alpha \cdot \dim T &= \dim \pazocal{G}_{T}\big(B^H\big) \leq \mathcal{P}^\alpha\m\dim \pazocal{G}_{T}\big(B^H\big) \\
 & = \mathcal{P}^\alpha\m\dim \pazocal{G}_{T}(f_0) \leq \alpha \cdot \dim T.
\end{align*}
In the other cases the same theorems $\Pb$-almost surely yield
\begin{align*}
& \dim T + (1 - 1/\alpha) \cdot d 
= \dim \pazocal{G}_{T}\big(B^H\big)
\leq \mathcal{P}^\alpha\m\dim \pazocal{G}_{T}\big(B^H\big)/\alpha + (1 - 1/\alpha) \cdot d\\
& \quad = \mathcal{P}^\alpha\m\dim \pazocal{G}_{T}(f_0) / \alpha + (1 - 1/\alpha) \cdot d 
\leq \dim T + (1 - 1/\alpha) \cdot d
\end{align*}
and the claim follows.
\end{remark}

The calculations in the proof of the previous theorem further show:

\begin{cor}\label{cor:constant_function}
Let $T\subseteq \R_+$ be a Borel set and $\alpha \in (0,\infty)$. Define the constant function $f_{C}(x)= C \in \Rd$ for all $x\in T$. Then one has
\begin{equation*}
\mathcal{P}^\alpha\m\dim \pazocal{G}_T( f_{C}) = (\alpha \vee 1) \cdot \dim T.
\end{equation*}
\end{cor}

As a consequence, we recover a  well-known result for the range of an isotropic $\alpha$-stable L\'{e}vy process; see \cite{BG60b} and Theorem 3.1 in \cite{MX05}.

\begin{cor}
Let $T\subseteq \R_+$ be a Borel set and $\alpha \in (0,2]$. Let $X=(X_t)_{t\geq0}$ be an isotropic $\alpha$-stable L\'{e}vy process. One $\Pb$-almost surely has
\begin{align*}
\dim \pazocal{R}_T(X) = \left(\alpha \cdot \dim T\right) \wedge d.
\end{align*}
\end{cor}

\begin{proof}
From Theorem \ref{thm:Main_theorem_range} and Corollary \ref{cor:constant_function} follows
\begin{align*}
\dim \pazocal{R}_T(X) = \left((\alpha \wedge 1) \cdot \mathcal{P}^\alpha\m\dim \pazocal{G}_T(f_0)\right) \wedge d = \left(\alpha \cdot \dim T\right) \wedge d,
\end{align*}
as claimed.
\end{proof}

We can also give some a priori estimates for the $\alpha$-parabolic Hausdorff dimension of the graph of a function in terms of the genuine Hausdorff dimension.

\begin{theorem}\label{thm:estimates_ph_alpha}
Let $T\subseteq \R$ be any set and $f:T \rightarrow \Rd$ be any function. Let $\ph_\alpha = \mathcal{P}^\alpha\m\dim \pazocal{G}_T(f)$. Then one has
\begin{equation}\label{eq:estimates_1}
\ph_\alpha \leq
\begin{cases}
\left(\ph_1 + \big(\frac{1}{\alpha} - 1\big)\cdot d\right) \ \wedge \ \left(d + 1\right), & \alpha \in (0,1],\\
\left(\ph_1 + \alpha - 1\right) \ \wedge \ \left(d + 1\right), & \alpha \in [1,\infty)
\end{cases}
\end{equation}
and
\begin{equation}\label{eq:estimates_2}
\ph_\alpha \geq
\begin{cases}
\ph_1 \ \vee \ \left(\frac{1}{\alpha} \cdot \ph_1 + 1 - \frac{1}{\alpha}\right), & \alpha \in (0,1],\\
\ph_1 \ \vee \ \left(\alpha \cdot \ph_1 + (1-\alpha) \cdot d\right), & \alpha \in [1,\infty).
\end{cases}
\end{equation}
Further, if $T \subseteq \R_+$ is a Borel set and $f : T \rightarrow \Rd$ is a Borel measurable function, then we obtain
\begin{equation}\label{eq:estimates_3}
\ph_\alpha \leq \left(\tfrac{1}{\alpha} \cdot \ph_1\right) \ \wedge \ \left(\ph_1 + \left(\tfrac{1}{\alpha} - 1 \right) \cdot d\right) \ \wedge \ \left(d+1\right), \quad \alpha \in (0,1].
\end{equation}
\end{theorem}

\begin{proof}
This follows immediately by Theorem \ref{le:General_Upper_Bound_Hausdorff} for (\ref{eq:estimates_1}) and (\ref{eq:estimates_2}) and Corollary \ref{cor:improvement_upper_bound} for (\ref{eq:estimates_3}) and the fact that the Hausdorff dimension never exceeds the toplogical dimension.
\end{proof}

Next we calculate some bounds for the parabolic Hausdorff dimension of $\beta$-H\"{o}lder continuous functions. These are functions $f:T \rightarrow \Rd$ that fulfil $\|f(t) - f(s)\| \leq C\cdot |t-s|^\beta$ for all $s,t\in T$ and some $\beta \in (0,1]$, $C>0$, denoted as $f \in C^\beta(T,\Rd)$. In case of $\alpha =1$, the following theorem is well-known, see \cite{K93}.

\begin{theorem}\label{thm:ph_Upper_bounds}
Let $T \subseteq \R$ be any set, $\alpha \in (0,\infty)$, $\beta \in (0,1]$ and $f\in C^\beta\big(T,\Rd\big)$. Define $\ph_\alpha := \mathcal{P}^\alpha\m\dim \pazocal{G}_T(f)$. Then one has
\begin{align*}
\ph_\alpha \leq
\begin{cases}
\left(\dim T + d \cdot \left(\frac{1}{\alpha} - \beta \right)\right) \ \wedge \ \frac{\dim T}{\alpha \beta}\ \wedge \ \left(d + 1\right), & \alpha \in (0,1],\\
\left(\alpha \cdot \dim T + d \cdot (1 - \alpha\beta)\right) \ \wedge \ \frac{\dim T}{\beta} \ \wedge \ \left(d + 1\right), & \alpha \in \big[1,\frac{1}{\beta}\big],\\
\left(\alpha \cdot \dim T\right)  \ \wedge \  \left(d + 1\right), & \alpha \in \big[\frac{1}{\beta},\infty\big).
\end{cases}
\end{align*}
\end{theorem}

\begin{proof}
Let $\tau > \dim T$ and $\eps > 0$ be arbitrary. Then we can cover $T$ by intervals $(\mathsf{T}_k)_{k\in\N}$ with sidelength $|\mathsf{T}_k| < 1$ such that $\sum_{k=1}^\infty |\mathsf{T}_k|^\tau < \eps$. Since $f \in C^\beta\big(T,\Rd\big)$, we can cover $\pazocal{G}_T(f)$ by $(\mathsf{B}_k)_{k \in \N} \subseteq \R^{1+ d}$ where
\begin{equation*}
\mathsf{B}_k := \mathsf{T}_k \times \prod_{i=1}^d \big[x_{i,k},\ x_{i,k} + C\cdot |\mathsf{T}_k|^\beta\big] \quad \text{ for every } k \in \N.
\end{equation*}
Note that without loss of generality we may assume $C\geq1$ for the constant in the definition of H\"older continuity.

\emph{(i)} Let $\alpha \in (0,1]$. On the one hand, for every $k\in\N$ we can cover $\mathsf{B}_k$ by (several) $\alpha$-parabolic cylinders with sidelength $|\mathsf{T}_k|$ in time. Since $K \cdot |\mathsf{T}_k|^{1/\alpha} \geq C\cdot |\mathsf{T}_k|^\beta$ iff $K \geq C\cdot |\mathsf{T}_k|^{\beta - 1/\alpha}$
for some hypercubes $\square_{k,l}$ with sidelength $|\mathsf{T}_k|^{1/\alpha}$ we find a cover
\begin{equation*}
\pazocal{G}_T(f) \subseteq \bigcup_{k=1}^\infty \bigcup_{l=1}^{\big\lceil C\cdot |\mathsf{T}_k|^{\beta - 1/\alpha} \big\rceil^d} \mathsf{T}_k \times \square_{k, l}
\end{equation*}
with $\mathsf{T}_k \times \square_{k, l} \in \mathcal{P}^\alpha$ and $|\mathsf{T}_k \times \square_{k, l}|\asymp |\mathsf{T}_k|$ for every $k,l \in \N$. Now, for $\gamma = \tau + d \cdot ( 1/\alpha -\beta )$ we have
\begin{equation*}
\mathcal{P}^\alpha\m\pazocal{H}^\gamma ( \pazocal{G}_T(f)) \lesssim \sum_{k=1}^\infty |\mathsf{T}_k|^{d \cdot (\beta-1/\alpha) + \gamma} = \sum_{k=1}^\infty |\mathsf{T}_k|^\tau < \eps.
\end{equation*}
Since $\tau > \dim T$ is arbitrary, $\ph_\alpha \leq \dim T + d \cdot ( 1/\alpha - \beta).$
On the other hand, for every $k\in\N$ we can cover $\mathsf{B}_k$ by a single $\alpha$-parabolic cylinder with sidelength $C^{\alpha}\cdot |\mathsf{T}_k|^{\alpha\beta}$ in time. Then $\pazocal{G}_T(f) \subseteq \bigcup_{k\in\N}\ \mathsf{P}_{k}$ with $|\mathsf{P}_k|\asymp |\mathsf{T}_k|^{\alpha\beta}$.
Now, for $\gamma = \tau/(\alpha\beta)$ we have $\mathcal{P}^\alpha\m\pazocal{H}^\gamma ( \pazocal{G}_T(f)) \lesssim \sum_{k=1}^\infty |\mathsf{T}_k|^{\alpha\beta \cdot \gamma} < \eps$. Since $\tau > \dim T$ is arbitrary this results in $\ph_\alpha \leq \frac{\dim T}{\alpha\beta}$.\\
\emph{(ii)} Let $\alpha \in [1,1/\beta]$. On the one hand, for every $k\in\N$ we can cover $\mathsf{B}_k$ by (several) $\alpha$-parabolic cylinders with sidelength $|\mathsf{T}_k|$ in time. The covering sets from part \emph{(i)} now fulfill $|\mathsf{T}_k \times \square_{k, l}|\asymp |\mathsf{T}_k|^{1/\alpha}$ and 
for $\gamma = \alpha \cdot \tau + d \cdot (1 - \alpha\beta) $ we have
\begin{equation*}
\mathcal{P}^\alpha\m\pazocal{H}^\gamma ( \pazocal{G}_T(f)) \lesssim \sum_{k=1}^\infty |\mathsf{T}_k|^{d \cdot (\beta-1/\alpha) + \gamma/\alpha}=\sum_{k=1}^\infty |\mathsf{T}_k|^\tau< \eps.
\end{equation*}
Since $\tau > \dim T$ is arbitrary, this results in $\ph_\alpha \leq \alpha \cdot \dim T + d \cdot (1 - \alpha\beta)$.
On the other hand, as in part \emph{(i)}  for every $k\in\N$ we can cover $\mathsf{B}_k$ by a single $\alpha$-parabolic cylinder with sidelength $C^{\alpha}\cdot |\mathsf{T}_k|^{\alpha\beta}$ in time. Then the cover $\pazocal{G}_T(f) \subseteq \bigcup_{k\in\N}\ \mathsf{P}_{k}$ now fulfills $|\mathsf{P}_{k}|\asymp |\mathsf{T}_{k}|^\beta$ and
for $\gamma = \tau/\beta$ we have $\mathcal{P}^\alpha\m\pazocal{H}^\gamma ( \pazocal{G}_T(f)) \lesssim \sum_{k=1}^\infty |\mathsf{T}_k|^{\beta \cdot \gamma} < \eps$. Since $\tau > \dim T$ is arbitrary, this results in $ \ph_\alpha \leq \dim T / \beta$.\\
\textit{(iii)} Let $\alpha \in [1/\beta,\infty)$. 
For every $k\in\N$ we can cover $\mathsf{B}_k$ by a single $\alpha$-parabolic cylinder $\mathsf{P}_k$ with length $C^{\alpha}\cdot |\mathsf{T}_k|$ in time. Then $\pazocal{G}_T(f) \subseteq \bigcup_{k=1}^\infty \mathsf{P}_k$ 
with $|\mathsf{P}_k\asymp |\mathsf{T}_k|^{1/\alpha}$ and
for $\gamma \geq \alpha \cdot \tau$ we have
$
\mathcal{P}^\alpha\m\pazocal{H}^\gamma ( \pazocal{G}_T(f)) \lesssim \sum_{k=1}^\infty |\mathsf{T}_k|^{\gamma / \alpha} < \eps.
$
Since $\tau > \dim T$ is arbitrary, this results in $\ph_\alpha \leq \alpha \cdot \dim T$.
\end{proof}

Let us inspect the important case $\alpha =2$, i.e.\ we aim to get a bound for the Hausdorff dimension of the graph of Brownian motion plus $\beta$-Hölder continuous drift function over $T$.

\begin{cor}\label{cor:B+f_Upper_Bounds}
Let $T\subseteq \R_+$ be any set. Let $B=(B_t)_{t\geq0}$ denote the $d$-dimensional Brownian motion and let $f \in C^\beta\big(T,\Rd\big)$ for some $\beta \in (0,1]$. Then one $\Pb$-almost surely has
\begin{equation*}
\dim \pazocal{G}_T(B + f) \leq
\begin{cases}
d + \frac{1}{2}, & \beta \leq \frac{\dim T}{d} \wedge \frac{1}{2} \wedge \left(\dim T - \frac{1}{2}\right),\\
\dim T + d \cdot (1 - \beta), 	& \dim T - \frac{1}{2} \leq \beta \leq \left(\frac{\dim T}{d} \wedge \frac{1}{2}\right),\\
\frac{\dim T}{\beta},			& \frac{\dim T}{d} \leq \beta \leq \frac{1}{2},\\
\left(2 \cdot \dim T\right) \wedge \left(\dim T + \frac{d}{2}\right),					&  \beta \geq \frac{1}{2}.
\end{cases}
\end{equation*}
Moreover, one $\Pb$-almost surely has
\begin{equation*}
\dim \pazocal{R}_T(B + f) \leq
\begin{cases}
\frac{\dim T}{\beta},		& \frac{\dim T}{d} \leq \beta \leq \frac{1}{2},\\
\left(2 \cdot \dim T\right) \wedge d,	& \beta \geq \frac{1}{2},\\
d,					& \text{else}.
\end{cases}
\end{equation*}
\end{cor}

\begin{proof}
Let $\ph_2 = \mathcal{P}^2\m\dim \pazocal{G}_T (f)$. Corollary \ref{cor_Upper_bound} $\Pb$-almost surely yields
\begin{equation*}
\dim \pazocal{G}_T(B+f) \leq
\ph_2 \ \wedge \ \frac{\ph_2 + d}{2}
\end{equation*}
and Theorem \ref{thm:Range_dim_X+f_leq_dim_f} $\Pb$-almost surely yields
\begin{equation*}
\dim \pazocal{R}_T(B+f) \leq \ph_2 \ \wedge \ d.
\end{equation*}

Finally, by Theorem \ref{thm:ph_Upper_bounds} we easily get
\begin{equation*}
\ph_2 \leq
\begin{cases}
\left(2 \cdot \dim T + d\cdot (1 - 2 \beta)\right) \wedge \left(d + 1\right), 	& \beta \leq \frac{\dim T}{d} \wedge \frac{1}{2},\\
\frac{\dim T}{\beta}, 						& \frac{\dim T}{d} \leq \beta \leq \frac{1}{2},\\
2 \cdot \dim T, 							& \beta \geq \frac{1}{2}
\end{cases}
\end{equation*}
and the claim follows.
\end{proof}

\bibliographystyle{plain}

\end{document}